\documentclass{article}
\usepackage[utf8]{inputenc}
\usepackage[english]{babel}
\usepackage{amsmath, amsthm, amssymb}
\usepackage{mathtools}
\usepackage{enumitem}
\usepackage{commath}
\usepackage{amssymb}
\usepackage{nccmath}
\usepackage{xcolor}
\usepackage{hyperref}
\hypersetup{
    colorlinks=true,
    linkcolor=blue,
    filecolor=magenta,      
    urlcolor=cyan,
    citecolor=red
}
\usepackage{cleveref} 
\usepackage[margin=1.5in]{geometry}

\usepackage{tikz}
\usetikzlibrary{calc}
\usetikzlibrary{patterns}
\usepackage{float}
\usepackage[font=small,labelfont=bf]{caption}

\newtheorem{defn}{Definition}

\newtheorem{remark}{Remark}
\newtheorem{theorem}{Theorem}
\newtheorem{lemma}{Lemma}
\newtheorem{cor}{Corollary}
\newtheorem{example}{Example}

\crefname{lemma}{Lemma}{Lemmas}

\newcommand{\cu}{\operatorname{curl}}

\newcommand{\Hmm}[1]{\leavevmode{\marginpar{\tiny%
$\hbox to 0mm{\hspace*{-0.5mm}$\leftarrow$\hss}%
\vcenter{\vrule depth 0.1mm height 0.1mm width \the\marginparwidth}%
\hbox to
0mm{\hss$\rightarrow$\hspace*{-0.5mm}}$\\\relax\raggedright #1}}}
\newcommand{\di}{\operatorname{div}}

\newcommand{\R}{{\mathbb{R}}}
\newcommand{\N}{{\mathbb{N}}}

\newcommand{\xn}{X_{\rm \scriptscriptstyle N}}

\title{Spectral stability of the $curl curl $ operator via uniform Gaffney inequalities on perturbed  electromagnetic cavities }
\author{Pier Domenico Lamberti\footnote{corresponding author }  and Michele Zaccaron}

\begin{document}

\maketitle

\noindent
{\bf Abstract:}  We prove  spectral stability results for the $curl curl$ operator subject to electric boundary conditions on a cavity upon boundary perturbations. The cavities are assumed to be sufficiently smooth but we impose weak restrictions on the strength of the perturbations. 
The methods are of variational type and  are based on two main ingredients: the construction of suitable Piola-type transformations between domains and 
the proof of uniform Gaffney inequalities obtained by means of uniform a priori  $H^2$-estimates for the Poisson problem of the Dirichlet Laplacian. The uniform a priori estimates are proved by using the results of V. Maz'ya and T. Shaposhnikova  based on Sobolev multipliers. Connections to  boundary  homogenization problems are also indicated. 

\vspace{11pt}


\vspace{6pt}

\section{Introduction}

In this paper we study the spectral stability of the $curl curl $ operator on an electromagnetic  cavity $\Omega$  in $\R^3$  upon  perturbation of the shape of $\Omega$. 
The cavity  $\Omega$ is  a bounded  connected open set (shortly, a bounded domain), the boundary of which is enough regular  to 
guarantee the validity of  the celebrated Gaffney inequality, that is 
\begin{equation}\label{gaffintro}
\norm{u}_{H^1(\Omega)^3} \le C\left( \norm{u}_{L^2(\Omega)^3} + \norm{\cu u}_{L^2(\Omega)^3} + \norm{\di u}_{L^2(\Omega)}  \right)
\end{equation}
for all vector fields $u\in L^2(\Omega)^3$ with distributional $\cu u  \in L^2(\Omega)^3$ and  $\di u \in L^2(\Omega)$, and satisfying the so-called electric boundary conditions
$$
\nu \times u=0, \ {\rm on}\ \partial\Omega \, .
$$
Here $\nu $ denotes the unit outer normal to $\partial\Omega$ and $H^1(\Omega)$ is the standard Sobolev space of functions in $L^2(\Omega)$ with first order weak derivatives in $L^2(\Omega)$. 
It is classical that the   Gaffney inequality holds for domains $\Omega$ with  boundaries of class  $C^{2}$, but the regularity can be relaxed in order to 
include boundaries of class $C^{1,\beta}$ with $\beta >1/2$, see  \cite{fil, PrFil15}. 

 The eigenvalue problem under consideration is 
\begin{equation}\label{main}
\left\{
\begin{array}{ll}
\cu\, \cu u = \lambda u,& \ \ {\rm in}\ \Omega ,\\
\nu \times u = 0 ,& \ \ {\rm on}\ \partial \Omega ,
\end{array}
\right.
\end{equation}
 and is immediately derived from  the  time-harmonic Maxwell's equations 
%in a homogeneous isotropic medium filling a domain $\Omega$ in $\R^3$   can be written  as 
\begin{equation}\label{Maxwelltimeharm}
\cu E - {\rm i} \, \omega  \, H = 0\,,\,\,\cu H + {\rm i} \, \omega  \, E = 0\, ,
\end{equation}
where $E, H$ denote the spatial parts of the electric and the magnetic field respectively and  $\omega  > 0$ is the angular frequency.  
Indeed, taking the curl in the first equation of  
\eqref{Maxwelltimeharm} and setting $\lambda =\omega^2$, one immediately obtains problem \eqref{main}. Note that here the medium filling $\Omega$ is homogeneous and isotropic and for simplicity  the corresponding electric permittivity $\varepsilon$  and magnetic permeability   $\mu$   have been normalized by setting $\varepsilon =\mu =1$.
%Note that the solutions $E$, $H$ to \eqref{Maxwelltimeharm} are  divergence-free.  
The boundary conditions are   those of a perfect conductor,  namely $\nu \times E=0 \ \ {\rm and }\ \  H\cdot \nu =0$. Thus, the vector field $u$ in \eqref{main} plays the role of the electric field $E$ (similarly, the magnetic field would satisfy the same equation but with the other boundary conditions $u\cdot \nu =0$ and $\nu \times \cu\,  u =0 $).

We observe that   the study of electromagnetic cavities is quite important in applications, for example in designing cavity resonators or shielding structures for electronic circuits,   see e.g.,  \cite[Chp.~10]{hanyak}.  We also refer to  
  \cite{ces, kihe, monk, ned, rsy} for details and references concerning the mathematical theory of electromagnetism.  See also  \cite{co91, cocoercive, coda, coda2, lamstra, pauly, yin}.

The spectrum of problem \eqref{main} is discrete and consists of a divergent sequence of positive  eigenvalues $\lambda_n[\Omega]$ of finite multiplicity. 

In this paper, we study the dependence of $\lambda_n[\Omega]$ and the corresponding eigenfunctions  upon variation of $\Omega$. 
It seems to us that  very little is known in the literature.  The case of domain perturbations of the form $\Phi (\Omega)$ where $\Phi $ is a regular diffeomorphism 
from $\Omega$ to $\Phi (\Omega )$ are considered in \cite{jimbo} and \cite{lamzac} where differentiability results and Hadamard type formulas for shape derivatives are proved.  We also quote the pioneering work \cite{hira}  where the Hadamard formula was found  on the base of heuristic computations.  We  note that shape derivatives are used in inverse electromagnetic scattering in  \cite{he1, he2, he3}. 

The aim of the present paper is to prove spectral  stability results under less stringent assumptions on the families of domain perturbations. For this purpose, we adopt the approach of  \cite{arrlam} further developed in \cite{arrferlam, ferra, ferralamb, ferlam}. 

Given a fixed domain $\Omega$, we consider a family of domains $\Omega_{\epsilon}$, $\epsilon>0$ converging to $\Omega$  as $\epsilon \to 0$. The convergence of $\Omega_{\epsilon}$ to $\Omega$ will be described by means of a fixed atlas ${\mathcal{A}}$, that is a finite collection of rotated parallelepipeds 
$V_j$, $j=1,\dots , s$  covering the domains under consideration and such that if $V_j$ touches the boundaries of the domains then $\Omega\cap V_j$ and $\Omega_{\epsilon}\cap V_j$ are given by the subgraphs of two functions $g_j, g_{\epsilon ,j}$ in two variables, say $\bar x= (x_1,x_2)$.  Thus the convergence of $\Omega_{\epsilon}$ to $\Omega$ is understood in terms of the convergence of $g_{\epsilon ,j}$ to $g_j$ as $\epsilon \to 0$. 

It is not surprising that if $g_{\epsilon ,j}$ converges uniformly to $g_j$ together with its first and second derivatives as $\epsilon \to 0$ (in which case one talks of $C^2$-convergence) then we have spectral stability of the $curl curl $ operator, which means that the  eigenvalues and eigenfunctions of the problem in $\Omega_{\epsilon}$ converge to those in $\Omega$ as $\epsilon \to 0$. It is also not surprising that if $g_{\epsilon ,j}$ converges uniformly to $g_j$ together with its first derivatives 
and 
\begin{equation}\label{unic2}
\sup_{\epsilon >0}\sup_{ \bar x\in \R^2} |D ^2g_{\epsilon ,j } (\bar x )| \ne \infty 
\end{equation}
then we have spectral stability again.  (These results are also immediate consequences of the results of the present paper.) The main  question here  is whether it is possible to relax condition \eqref{unic2}.  For example, if we assume that $g_{\epsilon , j}$ is of the form 
\begin{equation}\label{model}
g_{\epsilon , j} =\epsilon^{\alpha }b_j(\bar x / \epsilon )
\end{equation}
where $\alpha >0$ and $b_j$ is a fixed $C^{1,1}$ function, condition $\eqref{unic2}$ is encoded by the inequality $\alpha \geq 2$. In this model case, the question
is whether one can get spectral stability for $\alpha <2$.  Note that a profile of the form \eqref{model} is typical in the study of boundary homogenization problems 
and thin domains, see for example \cite{arrferlam, arrlam, arrvil, capevi, casado, ferralamb, ferlam, ferlam2}.

This problem was solved for the biharmonic operator with intermediate boundary conditions (modelling an elastic hinged plate) in \cite{arrlam} where condition \eqref{unic2} is relaxed by introducing a  suitable notion of weighted convergence which allows to prove spectral stability for $\alpha >3/2$ in the model problem above. That condition is described here in \eqref{assumptions}.   It is remarkable that the threshold $3/2$ is sharp since for $\alpha \le 3/2 $ spectral stability does not occur for the probelm discussed in \cite{arrlam} (in particular, it is proved in \cite{arrlam} that for $\alpha <3/2$ a degeneration phenomenon occurs and for $\alpha =3/2$ a strange term in the limit appears, as in many homogenization problems). An analogous  trichotomy  is found in \cite{ferlam} for the biharmonic operator subject to certain Steklov type boundary conditions. 

In this paper, we prove that the relaxed  convergence  \eqref{assumptions}  guarantees the spectral stability of the $curl curl $ operator. Our result requires that the Gaffney inequality  \eqref{gaffintro}
 holds for all domains $\Omega_{\epsilon}$ with a constant $C$ independent of $\epsilon$. Again, if one does not assume the validity of 
the uniform bound \eqref{unic2}, then proving that a uniform  Gaffney inequality holds  is highly non-trivial.  Here we manage to do this, by exploting 
the approach of   \cite[Ch.~14]{Mazya}  based on the use of Sobolev multipliers and the notion of domains of class $\mathcal{M}^{3/2}_2(\delta)$. 
In particular, if we assume that 
\begin{equation}\label{unic32}
 |\nabla g_{\epsilon,j } (\bar x) -\nabla g_{\epsilon,j } (\bar y)|\le M |\bar x-\bar y|^{\beta }
\end{equation}
for all $ \bar x, \bar y\in \R^2$, with $\beta \in ]1/2,1]$ and $M$ independent of $\epsilon$, and we also assume that the $\sup$-norms of functions $|\nabla g_{\epsilon, j } |$  are sufficienlty small, then our domains belong to the class $\mathcal{M}^{3/2}_2(\delta)$  with $\delta $ small enough. This allows  to apply \cite[Thm.~14.5.1]{Mazya}
which guarantees the validity of  a uniform $H^2$-\,a priori estimate for the Dirichlet Laplacian which, in turn,  is  equivalent to the uniform Gaffney inequality.

In conclusion, the convergence of the domains $\Omega_{\epsilon}$ to $\Omega$ in the sense of \eqref{assumptions} combined with the validity of \eqref{unic32} 
and the smallness of the gradients of the profile functions $g_{\epsilon, j}$ guarantees the spectral stability of the $curl curl $ operator. 
Note that, in principle, since $\Omega$ is of class $C^1$ one may think of choosing from the very beginning an atlas which guarantees that the gradients 
of the profile functions are as small as required (indeed, it is enough to adapt the atlas to the tangent planes of  a  sufficiently big number of boundary points of $\Omega$). Then the convergence in the sense of \eqref{assumptions} would imply the smallness of the gradients of the profile functions of $\Omega_{\epsilon}$ as well.

By setting $\beta = \alpha-1$, we deduce that  a uniform Gaffney inequality holds for the example provided by \eqref{model} if $\alpha > 3/2$. Moreover, if $\alpha >3/2$ then spectral stability occurs for the same example   since in this case also the  convergence \eqref{assumptions} occurs. 

The case $\alpha \le 3/2$ is more involved and we plan to address it in a forthcoming paper, see Remark \ref{finalremark}.
We note that if $\alpha <3/2$  one cannot expect the validity of uniform Gaffney inequalities, in particular because the regularity assumptions $C^{1,\beta}$ for  $\beta >1/2$  is optimal for the validity of the Gaffney inequality itself, see \cite{fil, PrFil15}.   

One of the main tools used in this paper is a Piola-type transform which allows to pull back  functions from $\Omega$ to $\Omega_{\epsilon}$ preserving 
the boundary conditions.  In particular the transformation depends on $\epsilon$ and  is constructed in such a way that  for any fixed  compact set $K$ contained in $\Omega\cap\Omega_{\epsilon}$, it  does not modify the values of the  vector fields on $K$ for $\epsilon$ sufficiently small.  Our Piola transform is constructed by 
pasting together local Piola transforms defined in each local chart of the atlas and for this reason it is called here {\it Atlas Piola transform}.  We believe that our construction has its own interest. 

This paper is organized as follows. Section 2 is devoted to preliminaries and notation concerning the atlas classes, the functions spaces and the weak formulations
of our problems. Section 3 is devoted to the construction of the Atlas Piola transform and to the proof of its main properties, see Theorem~\ref{Piolamain}. 
In Section 4 we prove our main stability theorem, namely Theorem~\ref{principale}. 
Section 5 is devoted to the proof 
of uniform a priori estimates and uniform Gaffney inequalities - see Corollaries~\ref{aprioricorol}, \ref{aprioricorol2} -   and  contains the corresponding applications to the spectral stability problems, see Theorems~\ref{principaleholder}, \ref{mainhomo}.

\section{Preliminaries and notation}

\subsection{Classes of open sets}

In this paper we consider open sets $\Omega$ in $\R^N$, in particular  in $\R^3$, with sufficiently regular boundaries. This means that $\Omega$ can be described in a neighborhood of any point of the boundary as the subgraph  of a sufficiently regular  function $g$ defined in a local system of orthogonal coordinates. The regularity of $\Omega$
depends on the regularity of the functions $g$.  Since we aim at studying domain perturbation problems, following \cite{burlam} and \cite{ferlam}, we find convenient 
to use the notion of {\it atlas}, that is a collection ${\mathcal{A}}$ of rotated parallelepipeds $V_j$, $j=1, \dots , s$,  which cover $\Omega$ and such that if $V_j $ touches the boundary of $\Omega$ then $\Omega\cap V_j$ is a subgraph of a function $g_j$.  The parallelepipeds will also be called {\it local charts}.
More precisely, in Definition \ref{atlas} below the atlas $\mathcal{A}$ is defined as $(\rho,s,s',\{V_j\}_{j=1}^s, \{r_j\}^s_{j=1})$ where $s$ is the total number of cuboids used to cover $\Omega$, $s'$ is the number of cuboids touching the boundary of $\Omega$, $V_j$ are the cuboids, $r_j$ are the rotations used to change variables in the representations of the local charts, and $\rho$ is a parameter controlling the minima and maxima of the functions $g_j$.
Note that in this paper the atlas ${\mathcal{A}}$ will be often fixed, while the functions $g_j$, hence $\Omega$, will be perturbed.

Given a set $V\subset \mathbb{R}^N$ and  a parameter $\rho >0$, we write
$V_\rho :=\set{x \in V : d(x, \partial V) > \rho}$.
\begin{defn} \label{atlas}
Let $\rho>0, s,s' \in \mathbb{N}, s' \leq s$ and $\{V_j\}_{j=1}^s$ be a family of bounded open cuboids (i.e. rotations of rectangle parallelepipeds in $\mathbb{R}^N$) and $\{r_j\}^s_{j=1}$ be a family of rotations in $\mathbb{R}^N$. We say that $\mathcal{A}=(\rho,s,s',\{V_j\}_{j=1}^s, \{r_j\}^s_{j=1}$) is an atlas in $\mathbb{R}^N$ with parameters $\rho,s,s',\{V_j\}_{j=1}^s, \{r_j\}^s_{j=1}$, briefly an atlas in $\mathbb{R}^N$.

A bounded domain $\Omega \subset \mathbb{R}^N$ is said to be of  class $C^{k,\gamma}_M(\mathcal{A})$ with $k \in \mathbb{N} \cup \{0\}$, $\gamma \in [0,1]$ and $M>0$ if it satisfies the following conditions:
\begin{enumerate}[label=(\roman*),font=\upshape]
\item $\Omega \subset \bigcup_{j=1}^s (V_j)_\rho$ and $(V_j)_\rho \cap \Omega \neq \emptyset$;
\item $V_j \cap \partial\Omega \neq \emptyset$ for $j=1,\dots,s'$ and $V_j \cap \partial \Omega = \emptyset$ for $s'+1 \leq j \leq s$;
\item for $j=1,\dots,s$ we have
$$r_j(V_j)=\set{x \in \mathbb{R}^N : a_{ij} < x_i < b_{ij}, i=1,\dots , N},$$
for $j=1,\dots,s'$ we have
$$r_j(V_j \cap \Omega) = \set{x=(\bar{x},x_N) \in \mathbb{R}^N : \bar{x}\in W_j, a_{Nj}<x_N<g_j(\bar{x})},$$
where $\bar{x}=(x_1,x_2)$,
$$W_j=\set{\bar{x}\in \mathbb{R}^{N-1}, a_{ij}<x_i<b_{ij}, i=1,\dots , N-1}$$
and the functions $g_j \in C^{k,\gamma}(\overline{W_j})$ for any $j=1,\dots, s'$. Moreover, for $j=1,\dots,s'$  
$$a_{Nj}+\rho\leq g_j(\bar{x})\leq b_{Nj} - \rho$$
for all $\bar{x}\in \overline{W_j}$.
\item $$\sup_{|\alpha| \le k}\| D^{\alpha}g_j \|_{L^{\infty}(W_j)}+ \sup_{|\alpha | =k}\sup_{\substack{\bar x, \bar y \in W_j\\ \bar x\ne \bar y  }}\frac{|D^{\alpha}g_j(\bar x)-D^{\alpha}g_j(\bar y)|}{|\bar x -\bar y|^{\gamma }}\le M$$
for $j=1,\dots,s'$.
\end{enumerate}
We say that $\Omega$ is of class $C^{k,\gamma}(\mathcal{A}) $   if it is of class $C^{k,\gamma}_M(\mathcal{A}) $ for some $M>0$; we say that $\Omega$  is of class   $C^{k,\gamma}$ if it is of class  $C^{k,\gamma}(\mathcal{A})$ for some atlas $\mathcal{A}$.
\end{defn}

\subsection{Function spaces}

%In this paper, the vectors of $\mathbb{R}^3$ are understood  as row vectors.  The transpose of a matrix  $A$ is denoted by $A^T$, hence  if $a \in \mathbb{R}^3$, then $a^T$ is a column vector.  If $a,b \in \mathbb{R}^3$ are two vectors, we denote by $\cdot$ the usual scalar product, that is $a \cdot b = a b^T$. 
In this section, we recall basic facts and notation for the function spaces that will be used in the following. We refer e.g., to \cite[Ch.~2]{gira} for more details. 

Here by  $\Omega$ we denote a bounded domain - that is, a bounded connected open set -  in $\R^3$.  Since the differential problems under consideration are associated  with self-adjoint operators,  the space  $L^2(\Omega)$ is  understood  here as a space of real-valued functions and is endowed with the scalar product $\int_{\Omega}u\cdot v \, dx $ defined for all vector fields  $u,v\in L^2(\Omega)^3$.    The space of vector fields
$u\in  L^2(\Omega)^3 $ with distributional curl in  $L^2(\Omega)^3$ is denoted by   $H(\cu, \Omega)$  and is  endowed with the norm defined by 
$$ ||u||_{H(\cu, \Omega)} = \left( ||u||^{2}_{L^2(\Omega)^3 }    + ||\cu u||^{2}_{L^2(\Omega)^3 } \right)^{1/2} $$ 
for all $u\in H(\cu, \Omega)$.  The closure in $H(\cu, \Omega)$ of the space of ${\mathcal{C}}^{\infty}$-functions with compact support in $\Omega$ is denoted by 
$H_{0}(\cu, \Omega)$. The following lemma characterizes the space $H_0(\operatorname{curl},\Omega)$ and 
 is analogous to the well-known characterization of  the Sobolev space $H^1_0(\Omega)$ (see e.g., \cite{bre}). 
We include a short proof. 
Here by $v^0$ we denote the extension-by-zero of a vector field $v$, that is 
$$v^0 =
\begin{cases}
	v & \text{if }x \in \Omega,\vspace{1mm} \\
	0 & \text{if }x \in \mathbb{R}^3 \setminus \Omega\, .
\end{cases}
$$

\begin{lemma} \label{extbyzero}
Let $\Omega$ be a bounded open set of class $C^{0,1}$ and  $u \in H(\operatorname{curl},\Omega)$. Then 
 $u \in H_0(\operatorname{curl},\Omega)$ if and only if 
 $u^0\in H(\operatorname{curl},\mathbb{R}^3)$, in which  case $\operatorname{curl}(u^0) = (\operatorname{curl}u)^0$.
\end{lemma}

\begin{proof}
Suppose that $u^0$ belongs to $H(\operatorname{curl},\mathbb{R}^3)$. Thus, there exists $v \in (L^2(\mathbb{R}^3))^3$ such that 
\begin{equation} \label{weak:curl}
\int_ \Omega u \cdot \operatorname{curl}\varphi \, dx = \int_{\mathbb{R}^3} v \cdot \varphi \, dx \qquad \text{for all }\varphi \in (C^\infty_c(\mathbb{R}^3))^3.
\end{equation}
Since it holds in particular for all test functions $\varphi \in (C^\infty_c(\Omega))^3$, then necessarily $v=\operatorname{curl}u$ on  $\Omega$. On the other hand, since we can take any $\varphi \in (C^\infty_c(\mathbb{R}^3 \setminus \bar{\Omega}))^3,$ we see that $v=0$  outside $\Omega$. Hence we can rewrite \eqref{weak:curl} as follows
\begin{equation} \label{weak:curl:Omega}
\int_ \Omega u \cdot \operatorname{curl}\varphi \, dx = \int_{\Omega} \operatorname{curl}u \cdot \varphi \, dx \qquad \text{for all }\varphi \in (C^\infty_c(\mathbb{R}^3))^3.
\end{equation} 
By \cite[Lemma 2.4]{gira}  it follows that $u \in H_0(\operatorname{curl},\Omega)$. The converse implication is straighforward.
\end{proof}

We note that if $\Omega $ is sufficiently regular, say of class $C^{0,1}$, the space $H_0(\operatorname{curl}, \Omega)$ coincides with the set of square integrable vector fields in $\Omega$ whose curl is also square integrable, and such that their tangential trace at the boundary $\partial \Omega$ is zero (see \cite[Thm. 2.12]{gira}). In particular, we have that 
$$H_0(\operatorname{curl},\Omega) \cap (C^\infty(\bar{\Omega}))^3 = \set {u \in (C^\infty(\bar{\Omega}))^3 : \nu \times u\rvert_{\partial \Omega} =0},$$ 
where $C^\infty(\bar{\Omega})$ denotes smooth compactly supported functions of $\mathbb{R}^3$ restricted to $\bar{\Omega}$.

We denote by $ H(\di, \Omega) $ the space of vector fields
$u\in  L^2(\Omega)^3 $ with distributional divergence in  $L^2(\Omega)^3$, endowed with the norm defined by 
$$ ||u||_{H(\di, \Omega)} = \left( ||u||^{2}_{ L^2(\Omega)^3} + ||\di u||^{2}_{L^2(\Omega)} \right)^{1/2} $$
for all $u\in H(\di, \Omega)$. 
Finally, we set 
$$
   ||u||_{  X_{\rm \scriptscriptstyle  }(\Omega)  }\! =\!   \left( ||u||^{2}_{  L^2(\Omega)^3} + ||\cu u||^{2}_{L^2(\Omega)^3} + ||\di u||^{2}_{L^2(\Omega)} \right)^{1/2} ,
$$
and we consider the space 
$$
 X_{\rm \scriptscriptstyle N}(\Omega) := H_0(\cu, \Omega) \cap H (\di, \Omega)
$$
endowed  with the norm defined above, that is  $ ||u||_{  X_{\rm \scriptscriptstyle  N}(\Omega)  }\! =\!    ||u||_{  X_{\rm \scriptscriptstyle  }(\Omega)  }\! $
for all $u\in  X_{\rm \scriptscriptstyle  N}(\Omega) $.  
We also set  $X_{\rm \scriptscriptstyle  N}(\di 0, \Omega) := \{u \in X_{\rm \scriptscriptstyle  N}(\Omega) : \di u = 0 \,\, {\rm in}\ \Omega \}$.

%By $\di_{\scriptscriptstyle\partial \Omega}$ and $\gr_{\scriptscriptstyle\partial \Omega}$ we denote the usual tangential operators.\\

%A key difference between the spaces $H^1(\Omega)^3$ and $H(\cu, \Omega)$ is that the embedding of $H(\cu, \Omega)$ into $L^2(\Omega)^3$ is not compact (the same property holds as well for the intersections of $H(\cu, \Omega)$ with the kernels of the operators $\di$ and $\cu$ - and similarly for $H(\di, \Omega)$ in place of $H(\cu, \Omega)$). Hence, the determination of suitable spaces that are compactly embedded in $L^2(\Omega)^3$ is an important issue; it is known, \cite{web}, that for a bounded simply connected  domain of class ${\mathcal{C}}^{1,1}$, the space $X_{\rm \scriptscriptstyle\tiny T}(\Omega)$ is compactly embedded in $L^2(\Omega)^3$. In addition, $X_{\rm \scriptscriptstyle\tiny T}(\Omega)$  is continuously embedded in $H^1(\Omega)^3$. 

Recall that $H^1(\Omega)$ is the standard Sobolev space of functions in $L^2(\Omega)$ with first order weak derivatives in $L^2(\Omega)$. 
The celebrated Gaffney inequality allows to prove that the space $\xn (\Omega) $ is continuously embedded into the space $H^1(\Omega)^3$ provided 
 $\Omega$ is  sufficiently regular.  Namely, we have the following result, see e.g.  \cite[Theorem~3.7]{gira}.

\begin{theorem}\label{gaffneythm} Let $\Omega$ is  a bounded  open set in $\R^3$ of class $C^{1,1}$.  Then $ X_{\rm \scriptscriptstyle N}( \Omega)$ is continuously embedded into
$H^1(\Omega)^3$,  and there exists
$C>0$ such that the Gaffney inequality 
\begin{equation}\label{gaff}\| u\|_{H^1(\Omega)^3}\le C   ||u||_{  X_{\rm \scriptscriptstyle N}(\Omega)  },\end{equation}
holds for all $u\in   X_{\rm \scriptscriptstyle N}( \Omega)$. 
\end{theorem}

By the previous theorem it immediately follows that the space $\xn (\Omega) $ is compactly embedded into  $L^2(\Omega)^3$, since this is true for the space 
$H^1(\Omega)^3$.

As we shall see, the regularity assumptions on $\Omega$ in Theorem~\ref{gaffneythm} can be relaxed since the inequality holds for domains of class $C^{1, \beta}$ with $\beta \in ]1/2, 1]$, but some care is required, see Section~\ref{unifsec}.

\subsection{Weak formulations and resolvent operators}

Since for our purposes we prefer to work in the space $\xn (\Omega)$ rather than in the space $\xn(\di\,  0, \Omega) $, following \cite{cocoercive, coda},  we introduce  a penalty term
in the equation and  we replace problem  \eqref{main} by the problem
\begin{equation}\label{mainpen}
\left\{
\begin{array}{ll}
\cu\, \cu u  -\tau \nabla \di u = \lambda u,& \ \ {\rm in}\ \Omega ,\\
\di u =0,& \ \ {\rm on}\ \partial \Omega ,\\
\nu \times u = 0 ,& \ \ {\rm on}\ \partial \Omega ,
\end{array}
\right.
\end{equation}
where $\tau $ is any fixed positive real number.

It is easy to see that  problem \eqref{mainpen} can be formulated  in the weak sense as follows
\begin{equation}
\label{mainpenweak}
\int_{\Omega}\cu u \cdot \cu \varphi \, dx+\tau \int_{\Omega} \di u\,  \di \varphi \,dx=\lambda \int_{\Omega} u \cdot \varphi \, dx, \ \ {\rm for\ all}\ \varphi \in \xn (\Omega ),
\end{equation}
in the unknowns $u\in \xn (\Omega)$ and $\lambda \in \R$. Is obvious  that the  solutions  of  \eqref{main}  are exactly the divergence free solutions of \eqref{mainpenweak}. (Moreover, the weak formulation of \eqref{main} can be obtained simply by replacing $\xn (\Omega )$ by $\xn (\di 0, \Omega )$ in \eqref{mainpenweak}.)

Problem \eqref{mainpenweak}  admits also  solutions  which are not divergence free and which are given by the gradients of the solutions  to the Helmohltz 
equation with Dirichlet boundary conditions. Namely,   $u=\nabla f $   where $f$ solves the following problem
\begin{equation}\label{hel}
\left\{
\begin{array}{ll}
-\Delta f = \Lambda  f,& \ \ {\rm in}\ \Omega ,\\
f = 0 ,& \ \ {\rm on}\ \partial \Omega ,
\end{array}
\right.
\end{equation}
with $\Lambda =\frac{\lambda}{\tau}$. In fact, we have the following result from \cite{coda}.

\begin{lemma}\label{union} If  $\Omega$ is a  bounded domain in $\R^3$ of class ${\mathcal{C}}^{0,1}$, then  the set of  all eigenpairs $(\lambda , u)$ of  problem \eqref{mainpen} is the union  of the set of all  eigenpairs $(\lambda , u)$ of problem \eqref{main} and the set of all  eigenpairs of the form $(\tau \Lambda , \nabla f   )$  where $ (\Lambda , f)$ is an   eigenpair of problem \eqref{hel}.
\end{lemma}

Thus, we can directly study problem \eqref{mainpenweak} rather than the original problem \eqref{main}: this will always be understood in the following.
In fact, studying the spectral stability of problem  \eqref{mainpenweak} is equivalent to studying the spectral stability of problem \eqref{main} because the spurious eigenpairs introduced by the penalty term are given by the eigenpairs of the Dirichlet Laplacian which are stable for our class of domain perturbations (see \cite{arrlam}).

In order to study spectral stability problems, it is also convenient to recast the eigenvalue problems under consideration in the form of eigenvalue problems 
for compact self-adjoint operators and this can be done by passing to the analysis of the corresponding resolvent operators. A direct way of doing so, consists 
in defining  the operator $T$ from $\xn (\Omega )$ to its dual $(\xn (\Omega))'$  by setting
\begin{equation}\label{operatort}
<Tu, \varphi>= \int_{\Omega}\cu u \cdot \cu \varphi \, dx+\tau \int_{\Omega} \di u\,  \di \varphi \, dx  ,
\end{equation}
for all $u,\varphi \in \xn (\Omega )$, and considering  the map 
$J$ from $L^2(\Omega)^3$ to $(\xn (\Omega))'$  defined by 
$$
<Ju, \varphi>=\int_{\Omega} u \cdot \varphi \, dx ,
$$
for all $u\in L^2(\Omega)^3$ and $\varphi  \in \xn (\Omega )$. 
By restricting $J$  to $\xn(\Omega)$ (and denoting the restriction by the same symbol $J$), and using the Riesz Theorem it turns out that the operator $T+J$ is a homeomorphism from $\xn (\Omega )$ to its dual. The inverse operator $(T+J)^{-1}$ will serve for our purposes as discussed above. In fact, the following theorem holds. 
 
\begin{lemma}\label{esse} If $\Omega$ is a bounded domain in $\R^3$ such that  the embedding $\iota$  of $\xn (\Omega )$  into  $L^2(\Omega)^3$ is compact, then the 
 operator $S_{\Omega}$ from $L^2(\Omega)^3$ to itself defined by
$$
S_{\Omega } u=\iota \circ (T+J)^{-1}\circ J
$$
 is a non-negative compact self-adjoint operator in  $L^2(\Omega)^3$ whose eigenvalues $\mu$  are related to the eigenvalues $\lambda $ of problem \eqref{mainpenweak}   by the equality $\mu =(\lambda +1)^{-1}$.
\end{lemma}

 By the previous lemma and standard spectral theory it follows that the spectrum  $\sigma (S_{\Omega})$ of $S_{\Omega}$ can be represented as $\sigma (S_{\Omega})=\{0\} \cup \{\mu_n (\Omega ) \}_{n\in \N} $, where 
$\mu_n(\Omega )$, $n\in \N$  is a decreasing sequence of positive eigenvalues of finite multiplicity,  which converges to zero.  Consequently, 
the eigenvalues of problem  \eqref{mainpenweak}  can be represented  by the sequence $\lambda_n(\Omega ) $, $n\in \N$ defined by $
\lambda_n(\Omega )= \mu_n^{-1}(\Omega )-1
$.  Moreover, the classical Min-Max Principle yields the following variational representation
 \begin{equation}\label{minmax}
\lambda_n(\Omega )=   \min_{ \substack{ V\subset \xn(\Omega )  \\ {\rm dim }V=n }  }\  \,  \max _{u\in V\setminus\{0\}  }
  \frac{  \int_{\Omega}  |\cu u|^2 dx + \tau  \int_\Omega |\di u |^2  dx }{\int_{\Omega}   |u|^2\, dx}.
\end{equation}

%%%%%%%%%%%%%%%%%%
%%%%%%%%%%%%%%%%%%
%%%%%%%%%%%%%%%%%%

\bigskip

\section{A Piola-type approximation of the identity}

Given two domains $\Omega$ and $\tilde \Omega$ in ${\mathbb{R}}^3$ and a diffeomorphism $\Phi :\tilde\Omega \to \Omega$ of class $C^{1,1}$, the standard 
way to pull-back vector fields from $ \xn(\Omega) $  to $\xn(\tilde\Omega) $ consists in using the (covariant) Piola transform defined by
\begin{equation}\label{pullback}
u(x)= \left((v \circ \Phi) \operatorname{D}\Phi \,\right)(x),\ \ {\rm for\ all }\  x\in \tilde\Omega ,
\end{equation}
for all $v\in \xn(\Omega)$, see e.g., \cite{monk}. 
In fact, it turns out that $v\in  H_0(\cu, \Omega)   $ if and only if  $u\in H_0(\cu, \tilde \Omega) $, in which case we have  
\begin{equation}
(\operatorname{curl} v ) \circ \Phi = \frac{\operatorname{curl} u \left(\operatorname{D} \Phi  \right)^{T}} {\operatorname{det} \left( \operatorname{D} \Phi  \right)} \, .
\label{changecurl0}
\end{equation}
Note that for functions $u,v$ in $H^1$ we also have
 \begin{equation}
 \label{changediv0}
 (\displaystyle \operatorname{div} v) \circ \Phi   = \frac{\operatorname{div} \left[ u (\operatorname{D}\Phi )^{-1} (\operatorname{D}\Phi  )^{-T} \operatorname{det}(\operatorname{D} \Phi  ) \right]}{\operatorname{det} (\operatorname{D} \Phi ) },
 \end{equation}
and in this case $v \in \xn(\Omega) \cap H^1(\Omega)^3$ if and only if $u \in \xn(\tilde \Omega) \cap H^1(\tilde \Omega)^3$. See \cite{lamzac} for more details. 
Unfortunately, given two domains $\Omega$ and $\tilde \Omega$, in general it is not possible to define explicitly  a diffeomorphism between $\Omega$ and $\tilde \Omega$ (even if it is known a priori that the two domains are diffeomorphic). Nevertheless, it is important  for our purposes to define an operator which allows to pass from 
 $ \xn(\Omega) $  to $\xn(\tilde\Omega) $ as the Piola transform does.  
 This can be done by assuming that $\Omega$ and $\tilde\Omega$ belong to the same atlas class and   using a partition of unity in order to paste together Piola transforms
 defined locally, as described in the following. 
Note that the specific choice of local Piola transforms reflects our need for a transformation close to the identity. 
 
Let $\mathcal{A}$  be a fixed atlas in ${\mathbb{R}}^3$ and let $\Omega, \tilde \Omega$ be two domains of class $C^{1,1}(\mathcal{A})$. Let $g_j,\tilde g_j$ be the profile functions of $\Omega$ and $\tilde \Omega$ as in Definition~\ref{atlas}. Assume that  $k \in ]0,+\infty[$ is such that 
\begin{equation}\label{basic}
k>\max_{j=1,\dots, s'}\| \tilde g_j-g_j \|_{\infty}, \ {\rm and }\      \tilde g_j-k>a_{3,j}+\rho,\ \forall  j  = 1,\dots,s'. 
\end{equation}
For any $j=1,\dots,s'$ we set 
\begin{equation}\label{hatgj}
\hat{g}_j:= \tilde g_j -k
\end{equation}  
and we  define the map $h_{j}: r_j(\overline{\tilde \Omega \cap V_j}) \to \mathbb{R}$ 
\begin{equation}\label{accaj}
h_{j}(\bar{x},x_3) := \left\{
        \begin{array}{ll}
            0, & \quad \text{if } a_{3j}\leq x_3\leq \hat{g}_{j}(\bar{x}),\\
            (\tilde g_{j}(\bar{x})-g_j(\bar{x})) \left(\frac{x_3-\hat{g}_{j} (\bar{x})}{\tilde g_{j}(\bar{x})- \hat{g}_{j} (\bar{x})}\right)^3, & \quad \text{if } \hat{g}_{j} (\bar{x})<x_3\leq \tilde g_{j}(\bar{x}),
        \end{array}
    \right.
\end{equation}
and  the map 
\begin{equation}
\label{fij}
\Phi_{j} : r_j(\overline{\tilde \Omega \cap V_j}) \to r_j(\overline{\Omega \cap V_j}), \qquad \Phi_{j}(\bar{x},x_3):=(\bar{x}, x_3 - h_{j}(\bar{x},x_3)).
\end{equation}
Note that $\Phi_{j}$ coincides with the identity map on the set 
\begin{equation}
\label{kappaj}K_{j} :=\set{(\bar{x},x_3) \in W_j \times ]a_{3j},b_{3j}[ \ : a_{3j}<x_3<\hat{g}_{j}(\bar{x})}.
\end{equation}

Finally, if   $s'+1\leq j\leq s$ we define $\Phi_{j}: r_j(\overline{V_j}) \to r_j(\overline{V_j})$ to be the identity map. 

Observe that since $h_{j} \in C^{1,1}(r_j(\overline{\tilde\Omega \cap V_j}))$, then $\Phi_{j}$ is of class $C^{1,1}$, and so is the following map
\begin{equation}\label{psij}\Psi_{j} : \overline{\tilde \Omega \cap V_j} \to \overline{\Omega \cap V_j}, \qquad \Psi_{j} : = r_j^{-1} \circ \Phi_{j} \circ r_j.
\end{equation}

An easy computation shows that if  
\begin{equation}\label{basic2}
k>\frac{3}{\alpha}\max_{j=1,\dots, s'}\| \tilde g_j-g_j \|_{\infty}
\end{equation}
for some constant $\alpha \in ]0,1[$ then 
\begin{equation}\label{basic3}
0<1-\alpha \leq \operatorname{det}(\operatorname{D}\Psi_{j}(x)) \leq 1+\alpha  \quad \text{for any } x\in \tilde\Omega \cap V_j.
\end{equation}
Let  $\{\psi_j\}_{j=1}^s$ be a $C^{\infty}$-partition of unity associated with the open cover $\{V_j\}_{j=1}^s$ of the compact set $\overline{\cup_{j=1}^s(V_{j})_{\rho}}$
 that is  $0\leq \psi_j \leq 1$, $\operatorname{supp}(\psi_j) \subset V_j$ for all $j=1,\dots,s$, and $\sum_{j=1}^s \psi_j \equiv 1$ in $\overline{\cup_{j=1}^s(V_{j})_{\rho}}$, in particular also in $\overline{\Omega\cup\tilde\Omega}$. Note that this is  a partition of unity is  independent of  $\Omega, \tilde \Omega$ in the atlas class under consideration.

 Since for any  $\varphi \in \xn(\Omega)$  we have $\varphi=\sum_{j=1}^s \varphi_j$ where $\varphi_j=\psi_j \varphi$, then  it is  natural to give the following definition (note that here we consider open sets of class $C^{1,1}$ hence the spaces $\xn$ are embedded into $H^1$). 
  
 \begin{defn}  Let ${\mathcal{A}}$ be an atlas in ${\mathbb{R}}^3$ and  $\Omega, \tilde\Omega$ be two domains of class $C^{1,1}({\mathcal{A}})$. Assume that $k>0$ satisfies \eqref{basic}, and $\{\psi_j\}_{j=1}^s$  is a partition of unity as above.  The {\em Atlas  Piola transform} from $\Omega$ to $\tilde \Omega$,  with parameters ${\mathcal{A}}$, $k$, and $\{\psi_j\}_{j=1}^s$,  is the map   from  $\xn(\Omega)$ to $\xn(\tilde\Omega)$ defined by 
\begin{equation} \label{defE}
\mathcal{P} \varphi:=\sum_{j=1}^{s'} \tilde{\varphi}_{j} + \sum_{j=s'+1}^s \varphi_j
\end{equation}
for all $\varphi \in \xn(\Omega)$, where
\begin{equation}
\tilde{\varphi}_{j}(x):= \left\{
\begin{array}{ll}\label{tildefij}
(\varphi_j \circ \Psi_{j}(x)) \operatorname{D}\Psi_{j}(x), & \text{if } x\in \tilde \Omega \cap V_j,\vspace{1mm}\\
0, & \text{if } x\in \tilde\Omega \setminus V_j,
\end{array} \right.
\end{equation}
for any $j=1,\dots,s'$.
\end{defn}
Note that $\mathcal{P}\varphi \in \xn(\tilde{\Omega})$ because $(\varphi_j \circ \Psi_{j}) \operatorname{D}\Psi_{j} \in \xn(\tilde{\Omega} \cap V_j)$ (observe that the support of $\varphi_j$ is compact in $V_j$), hence $\tilde{\varphi}_j \in \xn(\tilde{\Omega})$.
 
This Atlas  Piola  transform   will be used in this paper for  a family $\Omega_\epsilon, \epsilon>0$ of domains  of class $C^{1,1}(\mathcal{A})$, converging in some sense  to a domain $\Omega$ of class $C^{1,1}(\mathcal{A})$.  In this case, $\Omega_{\epsilon}$ will play the role of the domain $\tilde \Omega$ and 
the corresponding transformation  will allow us to pass from $\xn(\Omega)$ to $\xn(\Omega_{\epsilon})$.

Given a family of domains $\Omega_{\epsilon}$, $\epsilon >0$, and a fixed domain $\Omega$, all of class $C^{1,1}(\mathcal{A})$, we shall denote by $g_{\epsilon,j}$ and $g_j$ the corresponding profile functions  (defined on  $W_j$)  of $\Omega_{\epsilon}$ and $\Omega$ respectively, as in Definition~\ref{atlas}. 
 Following \cite{arrlam,ferlam}, we use a notion of convergence  for the open sets $\Omega_{\epsilon}$ to $\Omega$, which is expressed in terms of convergence  of the  the profile functions $g_{\epsilon,j}$ to $g_j$. Namely, we assume that for any $\epsilon>0$ there exists $\kappa_\epsilon>0$ such that for any $j\in \set{1,\dots,s'}$
\begin{equation} \label{assumptions}
\begin{split}
&(i) \hspace{9pt} \kappa_\epsilon>\max_{ j=1,\dots , s'  } \norm{g_{\epsilon,j} - g_j}_{L^\infty(W_j)};\\
&(ii) \hspace{4pt} \lim_{\epsilon \to 0} \kappa_\epsilon =0;\\
&(iii) \hspace{4pt} \lim_{\epsilon \to 0} \frac{  \max_{ j=1,\dots , s'  }  \norm{D^\beta(g_{\epsilon,j} -g_j)}_{L^\infty(W_j)}}{\kappa_\epsilon^{3/2 - |\beta|}}=0 \quad \text{ for all }\ \beta \in \mathbb{N}^3 \text{ with } |\beta| \leq 2.
\end{split}
\end{equation}

Note that if every function $g_{\epsilon,j}$ converges to $g_j$ uniformly together with the first  order derivatives and condition \eqref{unic2} is satisfied  (in particular, if the second order derivatives of  $g_{\epsilon,j}$ converge uniformly to those of  $g_j$)   then conditions \eqref{assumptions} are fulfilled, see \cite{arrlam}. 
Note also that the exponent $3/2$ in \eqref{assumptions} turns out to be optimal in the analysis of \cite{arrlam} and plays a crucial role for instance in  proving inequality \eqref{estimateAbis}.

We now fix a partition of unity $\{ \psi_j\}_{j=1}^s$ associated with the covering of cuboids of the atlas ${\mathcal{A}}$ as above, and independent of $\Omega_{\epsilon}$
and $\Omega$. We also choose $k=6 \kappa_{\epsilon}$ and we denote by ${\mathcal{P}}_{\epsilon}$  the Atlas Piola transform  from $\Omega$ to $\Omega_{\epsilon}$ (with parameters ${\mathcal{A}}$, $k$, $\{ \psi_j\}_{j=1}^s$ ). 
Note that conditions \eqref{basic}, \eqref{basic2} \eqref{basic3} are satisfied with $\alpha =1/2$ if $\epsilon$ is sufficiently small. 

In the  following, we shall denote by $\hat g_{\epsilon, j}$, $h_{\epsilon, j}$, $\Phi_{\epsilon ,j}$, $K_{\epsilon ,j}$,   $\Psi_{\epsilon ,j}$,  $\tilde \varphi_{\epsilon ,j}$ all quantities defined in \eqref{hatgj}, \eqref{accaj}, \eqref{fij}, \eqref{kappaj}, \eqref{psij}, \eqref{tildefij} respectively, with $\tilde \Omega = \Omega_{\epsilon}$ and $k=6\kappa_{\epsilon}$.

Then we can prove the following theorem. We note that in the proof, some technical issues related to pasting together functions defined in  different charts are treated  in the spirit of the arguments used in  \cite{ferlam} for the Sobolev spaces $H^2(\Omega)$.

\begin{theorem}\label{Piolamain}
Let $\Omega_\epsilon, \epsilon>0$,  and $\Omega$ be bounded domains of class $C^{1,1}(\mathcal{A})$. Assume that $\Omega_{\epsilon}$ converges to $\Omega$ as $\epsilon\to 0$ in the sense of  \eqref{assumptions}. Let ${\mathcal{P}}_{\epsilon}$  be the  Atlas Piola transform   from $\Omega$ to $\Omega_{\epsilon}$ defined for $\epsilon $ sufficiently small as above. Then the following statements hold:
\begin{enumerate}[label=(\roman*),font=\upshape]
\item for any $\epsilon>0$ the function ${\mathcal P}_\epsilon$ maps  $\xn(\Omega)$ to $\xn(\Omega_\epsilon)$  with continuity;
\item for any compact set ${\mathcal{K}}$ contained in $\Omega$ there exists $\epsilon_{\mathcal{K}}  >0 $ such that 
\begin{equation}\label{identity}
({\mathcal P}_\epsilon \varphi )(x)= \varphi (x),\ \ \forall x\in {\mathcal{K}}
\end{equation}
 for all $\epsilon \in ]0,\epsilon_{\mathcal{K}}[$ and   $\varphi \in \xn(\Omega)$;
\item   the limit  
\begin{equation}
\norm{{\mathcal P}_\epsilon \varphi}_{\xn (\Omega_\epsilon)} \xrightarrow[\epsilon \to 0]{} \norm{\varphi}_{\xn (\Omega)},
\end{equation}
holds for all $\varphi \in \xn(\Omega)$;
\item the limit 
\begin{equation}
\norm{{\mathcal P}_\epsilon \varphi    -\varphi }_{ X (\Omega_\epsilon \cap \Omega  )} \xrightarrow[\epsilon \to 0]{}  0 , 
\end{equation}
holds for all $\varphi \in \xn(\Omega)$. 
\end{enumerate}
\end{theorem}

\begin{proof} Let $\varphi\in \xn(\Omega)$ be fixed. Note that $\Omega$ is of class $C^{1,1}$ hence the Gaffney inequality holds and $\varphi\in H^1(\Omega)^3$. 
Moreover, $\varphi_j\in \xn(\Omega)$ for all $j=1,\dots , s'$ hence $\tilde \varphi_{\epsilon , j }$ belongs to $\xn(\Omega_{\epsilon } )$ for all $j=1,\dots , s'$. It follows that ${\mathcal P}_\epsilon \varphi \in
\xn(\Omega_\epsilon)$.  
The continuity of the operator  follows by standard calculus, the Gaffney inequality and formulas \eqref{changecurl0},  \eqref{changediv0}. Thus, statement (i) holds. 

 For any fixed compact set ${\mathcal{K}}$ contained in $\Omega$, since $\hat g_{\epsilon , j}$ converges uniformly to $g_j$, we have  
$${\mathcal{K}}\cap V_j\subset  r_j^{-1} (K_{\epsilon , j} ) $$ for all $j=1, \dots , s'$ and $\epsilon$ sufficiently small; this, combined with the fact that  $\Phi_{\epsilon ,j}$ coincides with the identity on $K_{\epsilon ,j }$, it follows that
 $\tilde \varphi_{\epsilon, j}=\varphi_j$ on ${\mathcal{K}}$  for all $\epsilon$ sufficiently small and \eqref{identity} follows. 

We now prove statement (iii).  We have to prove the following limiting relations:
\begin{eqnarray} 
& & \lim_{\epsilon \to 0} \int_{\Omega_\epsilon} \abs{ {\mathcal P}_\epsilon \varphi}^2 = \int_\Omega \abs{\varphi}^2, \label{limitEvarphi}\\
& &  \lim_{\epsilon \to 0} \int_{\Omega_\epsilon} \abs{\operatorname{curl} {\mathcal P}_\epsilon \varphi}^2 = \int_\Omega \abs{\operatorname{curl} \varphi}^2, \label{limitcurlEvarphi} \\
& & \lim_{\epsilon \to 0} \int_{\Omega_\epsilon} \abs{\operatorname{div}{\mathcal P}_\epsilon \varphi}^2 = \int_\Omega \abs{\operatorname{div}\varphi}^2.  \label{limitdivEvarphi}
\end{eqnarray}

%Note that throughout the proof $C$ will denote a positive constant independent of $\epsilon$ and  may vary from line to line. 

We begin by proving  \eqref{limitEvarphi}. To see this, it just suffices to show that 
\begin{equation}\label{integrale1}
\lim_{\epsilon \to 0} \int_{\Omega_\epsilon} \tilde{\varphi}_{\epsilon,j} \cdot \tilde{\varphi}_{\epsilon,h} = \int_\Omega \varphi_j \cdot \varphi_h,
\end{equation}
and
\begin{equation} \label{integrale2}
\lim_{\epsilon \to 0} \int_{\Omega_\epsilon} \tilde{\varphi}_{\epsilon,j} \cdot \varphi_i = \int_\Omega \varphi_j \cdot \varphi_i
\end{equation}
for any $j,h \in \set{1,\dots,s'}$ and $i\in \set{s'+1,\dots,s}$.
We will only show \eqref{integrale1}, since the computations to prove \eqref{integrale2} are similar. 
We will first see that 
\begin{equation} \label{tildephigoesto0}
\lim_{\epsilon \to 0} \int_{(\Omega_\epsilon \cap V_j) \setminus  r^{-1}_j(K_{\epsilon,j})} \abs{\tilde{\varphi}_{\epsilon,j}}^2 =0.
\end{equation}
Notice that for any $j \in \set{1,\dots, s'}$ we have $\abs{(\Omega \cap V_j) \setminus r_j^{-1}(K_{\epsilon,j})}\to 0$ as $\epsilon$ goes to 0.
Moreover, if $w\in \mathbb{R}^3$ is a vector, then $\abs{w \operatorname{D}\Psi_{\epsilon,j}}= \abs{w \operatorname{D}\Phi_{\epsilon,j}} \leq C \abs{w}$, since
$$
\operatorname{D}\Phi_{\epsilon,j}=\begin{pmatrix} 
1 & 0 & 0 \\
0 & 1 & 0\\
-\frac{\partial h_{\epsilon,j}}{\partial x_1} & -\frac{\partial h_{\epsilon,j}}{\partial x_2} & 1-\frac{\partial h_{\epsilon,j}}{\partial x_3} \\
\end{pmatrix}
$$
and the first derivatives of $h_{\epsilon,j}$ are all bounded due to the hypothesis on the functions $g_{\epsilon,j}$ (see also \eqref{estimatederh}). Note  that here and in what follows, by $c$ we denote a constant independent of $\epsilon $ which may vary from line to line. 
Then by using also \eqref{basic3}, we have
\begin{align*}
\int_{(\Omega_\epsilon \cap V_j) \setminus  r^{-1}_j(K_{\epsilon,j})} \abs{\tilde{\varphi}_{\epsilon,j}}^2 dy &= \int_{(\Omega_\epsilon \cap V_j) \setminus  r^{-1}_j(K_{\epsilon,j})} \abs{(\varphi_j \circ \Psi_{\epsilon,j}) \operatorname{D}\Psi_{\epsilon,j}}^2 dy \\
&\leq c \int_{(\Omega_\epsilon \cap V_j) \setminus  r^{-1}_j(K_{\epsilon,j})} \abs{\varphi_j \circ \Psi_{\epsilon,j}}^2 dy\\
&= c \int_{(\Omega \cap V_j) \setminus  r^{-1}_j(K_{\epsilon,j})} \frac{\abs{\varphi_j}^2}{\abs{\operatorname{det}(\operatorname{D}\Psi_{\epsilon,j}) \circ \Psi^{(-1)}_{\epsilon,j}}} \, dx\\
&\leq c  \int_{(\Omega \cap V_j) \setminus  r^{-1}_j(K_{\epsilon,j})} \abs{\varphi_j}^2 dx \xrightarrow[\epsilon \to 0]{}0.
\end{align*}
By \eqref{tildephigoesto0} we deduce that
\begin{equation} \label{tildephigoestophi}
\lim_{\epsilon \to 0} \int_{\Omega_\epsilon} \abs{\tilde{\varphi}_{\epsilon,j}}^2 = \int_\Omega \abs{\varphi_j}^2.
\end{equation}
Indeed, since $\Psi_{\epsilon,j}$ is the identity on $r_j^{-1}(K_{\epsilon,j}) \subset \Omega \cap \Omega_{\epsilon}$,
using \eqref{tildephigoesto0} yields 
$$\int_{\Omega_\epsilon} \abs{\tilde{\varphi}_{\epsilon,j}}^2 = \int_{r_j^{-1}(K_{\epsilon,j})} \abs{\tilde{\varphi}_{\epsilon,j}}^2 + \int_{(\Omega_\epsilon \cap V_j)\setminus r_j^{-1}(K_{\epsilon,j})} \abs{\tilde{\varphi}_{\epsilon,j}}^2 \xrightarrow[\epsilon \to 0]{}  \int_{\Omega \cap V_j} \abs{\varphi_j}^2 =   \int_\Omega \abs{\varphi_j}^2.$$
Observe now that
\begin{equation} \label{dottildephisplitintegral}
\begin{split}
&\int_{\Omega_\epsilon} \tilde{\varphi}_{\epsilon,j} \cdot \tilde{\varphi}_{\epsilon,h}= \int_{\Omega_\epsilon \cap V_j \cap V_h} \tilde{\varphi}_{\epsilon,j} \cdot \tilde{\varphi}_{\epsilon,h}\\
& \quad =\int_{r_j^{-1}(K_{\epsilon,j}) \cap r_h^{-1}(K_{\epsilon,h})} \tilde{\varphi}_{\epsilon,j} \cdot \tilde{\varphi}_{\epsilon,h} + \int_{(\Omega_\epsilon \cap V_j \cap V_h) \setminus (r_j^{-1}(K_{\epsilon,j}) \cap r_h^{-1}(K_{\epsilon,h}))} \tilde{\varphi}_{\epsilon,j} \cdot \tilde{\varphi}_{\epsilon,h}.
\end{split}
\end{equation}
It is obvious that 
\begin{equation} \label{dottildephigoestodotphi}
\lim_{\epsilon \to 0} \int_{r_j^{-1}(K_{\epsilon,j}) \cap r_h^{-1}(K_{\epsilon,h})} \tilde{\varphi}_{\epsilon,j} \cdot \tilde{\varphi}_{\epsilon,h}  = \lim_{\epsilon \to 0} \int_{r_j^{-1}(K_{\epsilon,j}) \cap r_h^{-1}(K_{\epsilon,h})} \varphi_j \cdot \varphi_h =\int_\Omega \varphi_j \cdot \varphi_h.
\end{equation}
Here and in the following we will make use of the identity
\begin{equation} \label{setidentity}
\begin{split}
&(\Omega_\epsilon \cap V_j \cap V_h) \setminus (r_j^{-1}(K_{\epsilon,j}) \cap r_h^{-1}(K_{\epsilon,h}))= \\ 
& \qquad [(\Omega_\epsilon \cap V_j \cap V_h) \setminus r_j^{-1}(K_{\epsilon,j})] \cup [(r^{-1}_j(K_{\epsilon,j}) \setminus r^{-1}_h (K_{\epsilon,h})) \cap V_h ].
\end{split}
\end{equation}
Observe that by \eqref{tildephigoesto0}, \eqref{tildephigoestophi} we get
\begin{equation} \label{dottildephigoesto'}
\begin{split}
&\abs{\int_{(\Omega_\epsilon \cap V_j \cap V_h) \setminus r_j^{-1}(K_{\epsilon,j})}  \tilde{\varphi}_{\epsilon,j} \cdot \tilde{\varphi}_{\epsilon,h}}  \\
& \qquad \leq \left( \int_{(\Omega_\epsilon \cap V_j) \setminus r_j^{-1}(K_{\epsilon,j})} \abs{\tilde{\varphi}_{\epsilon,j}}^2 \right)^\frac{1}{2} \left( \int_{\Omega_\epsilon \cap V_h} \abs{\tilde{\varphi}_{\epsilon,h}}^2 \right)^\frac{1}{2} \xrightarrow[\epsilon \to 0]{}0,
\end{split}
\end{equation}
and 
\begin{equation} \label{dottildephigoesto0''}
\begin{split}
&\abs{\int_{(r^{-1}_j(K_{\epsilon,j}) \setminus r^{-1}_h (K_{\epsilon,h})) \cap V_h } \tilde{\varphi}_{\epsilon,j} \cdot \tilde{\varphi}_{\epsilon,h}} \\
&\quad \leq \left( \int_{r_j^{-1}(K_{\epsilon,j})} \abs{\tilde{\varphi}_{\epsilon,j}}^2 \right)^\frac{1}{2} \left( \int_{(\Omega_\epsilon \cap V_h) \setminus r^{-1}_h(K_{\epsilon,h})} \abs{\tilde{\varphi}_{\epsilon,h}}^2 \right)^\frac{1}{2} \xrightarrow[\epsilon \to 0]{}0.
\end{split}
\end{equation}
Hence, by formula \eqref{setidentity}, we see that the second term of the sum in \eqref{dottildephisplitintegral} vanishes  as $\epsilon$ goes to zero, hence  we deduce the validity of \eqref{integrale1}  from \eqref{dottildephisplitintegral} and \eqref{dottildephigoestodotphi}.

We now prove \eqref{limitcurlEvarphi}.  Again, we need to  check that 
\begin{equation} \label{curlintegrale1}
\lim_{\epsilon \to 0} \int_{\Omega_\epsilon} \operatorname{curl} \tilde{\varphi}_{\epsilon,j} \cdot \operatorname{curl}\tilde{\varphi}_{\epsilon,h} = \int_\Omega \operatorname{curl}\varphi_j \cdot \operatorname{curl} \varphi_h
\end{equation}
for any $j,h \in \set{1,\dots,s'}$.  Note that 
$$
\operatorname{D}\Phi^{(-1)}_{\epsilon,j}=\frac{1}{\operatorname{det}(\operatorname{D}\Phi_{\epsilon,j})}\begin{pmatrix} 
1-\frac{\partial h_{\epsilon,j}}{\partial x_3} & 0 & 0 \\
0 & 1-\frac{\partial h_{\epsilon,j}}{\partial x_3} & 0 \\
\frac{\partial h_{\epsilon,j}}{\partial x_1} & \frac{\partial h_{\epsilon,j}}{\partial x_2} & 1 \\
\end{pmatrix} \circ \Phi^{(-1)}_{\epsilon,j},
$$
and recall that $\Psi_{\epsilon,j} = r_j^{-1} \circ \Phi_{\epsilon,j} \circ r_j$. Moreover,  by \eqref{changecurl0}  we have
$$\operatorname{curl} \tilde{\varphi}_{\epsilon,j} = \left( \operatorname{curl}\varphi_j \circ \Psi_{\epsilon,j} \right) (\operatorname{D}\Psi_{\epsilon,j})^{-T} \operatorname{det}\operatorname{D}(\Psi_{\epsilon,j}) \ \    {\rm on}\  \Omega_\epsilon \cap V_j$$
so that 
$$\abs{\operatorname{curl}\tilde{\varphi}_{\epsilon,j}} \leq c \abs{\operatorname{curl}\varphi_j \circ \Psi_{\epsilon,j}}.$$
Then, with computations analogous to those performed above, we get
\begin{equation} \label{curltildephigoesto0}
\lim_{\epsilon \to 0} \int_{(\Omega_\epsilon \cap V_j) \setminus  r^{-1}_j(K_{\epsilon,j})} \abs{\operatorname{curl}\tilde{\varphi}_{\epsilon,j}}^2 =0.
\end{equation}
It is also obvious that 
\begin{equation} \label{curltildephionKepsilon}
\lim_{\epsilon \to 0} \int_{r_j^{-1}(K_{\epsilon,j})} \abs{\operatorname{curl}\tilde{\varphi}_{\epsilon,j}}^2 = \int_\Omega \abs{\operatorname{curl}\varphi_j}^2
\end{equation}
and thus 
\begin{equation} \label{curltildephigoestolphi}
\lim_{\epsilon \to 0} \int_{\Omega_\epsilon} \abs{\operatorname{curl}\tilde{\varphi}_{\epsilon,j}}^2 = \int_\Omega \abs{\operatorname{curl}\varphi_j}^2.
\end{equation}
By using the same argument above,  formula \eqref{setidentity}  together with the new identities \eqref{curltildephigoesto0}, \eqref{curltildephionKepsilon} and \eqref{curltildephigoestolphi}, we obtain \eqref{curlintegrale1} .

Finally, we prove \eqref{limitdivEvarphi}. To do so, we need to prove that  
\begin{equation} \label{divintegrale1}
\lim_{\epsilon \to 0} \int_{\Omega_\epsilon} \operatorname{div} \tilde{\varphi}_{\epsilon,j} \operatorname{div}\tilde{\varphi}_{\epsilon,h} = \int_\Omega \operatorname{div}\varphi_j \operatorname{div}\varphi_h
\end{equation}
for any $j,h \in \set{1,\dots,s'}$. Here and in the rest of the proof, the vectors under consideration will be represented as follows:   $\varphi_j=(\varphi_j^1, \varphi_j^2, \varphi_j^3)$ and $\Psi_{\epsilon,j}=(\Psi_{\epsilon,j}^1, \Psi_{\epsilon,j}^2, \Psi_{\epsilon,j}^3)$.

Since $\varphi \in \xn(\Omega)$ and the Gaffney inequality holds on $\Omega$, it follows that  $\varphi \in H^1(\Omega)^3$. Thus,  recalling that $\tilde{\varphi}_{\epsilon,j}(x) = (\varphi_j \circ \Psi_{\epsilon,j}(x)) \operatorname{D}\Psi_{\epsilon,j}(x)$ for all $x \in \Omega_\epsilon \cap V_j$, it is possible to apply the chain rule and obtain that  
\begin{equation} \label{type:A:and:B}
\operatorname{div}(\tilde{\varphi}_{\epsilon,j})= \sum_{m,n,i=1}^3 \underbrace{\left( \frac{\partial \varphi_j^m}{\partial x_n}(\Psi_{\epsilon,j}) \frac{\partial \Psi_{\epsilon,j}^n}{\partial x_i} \frac{\partial \Psi_{\epsilon,j}^m}{\partial x_i}\right)}_{\text{type A}} + \sum_{m,i=1}^3 \underbrace{\varphi^m_j (\Psi_{\epsilon,j}) \frac{\partial^2 \Psi^m_{\epsilon,j}}{\partial x_i^2}}_{\text{type B}} \quad \text{in } \Omega_\epsilon \cap V_j,
\end{equation}
where the terms in the first sum are called of type A and the others are called terms of type B.

Recall that $h_{\epsilon,j}$ are the functions in \eqref{accaj} used to define the diffeomorphisms    $ \Phi_{\epsilon,j}$.     We observe that by the Leibniz rule we have 
$$D^\alpha h_{\epsilon,j}(x) = \sum_{0 \leq \gamma \leq \alpha} \binom{\alpha}{\gamma} D^\gamma \left( g_{\epsilon,j}(\bar{x}) - g_j(\bar{x}) \right) D^{\alpha - \gamma} \left( \frac{x_3 - \hat{g}_{\epsilon,j} (\bar{x})}{g_{\epsilon,j} (\bar{x})- \hat{g}_{\epsilon,j}(\bar{x})} \right)^3$$
hence by standard calculus (note that the denominator in the previous formula is the constant $k=6\kappa_{\epsilon}$) and \eqref{assumptions} we get 
\begin{equation}
\abs{D^{\alpha - \gamma} \left( \frac{x_3 - \hat{g}_{\epsilon,j} (\bar{x})}{g_{\epsilon,j} (\bar{x})- \hat{g}_{\epsilon,j}(\bar{x})} \right)^3} \leq \frac{c}{\abs{g_{\epsilon,j} (\bar{x}) - \hat{g}_{\epsilon,j} (\bar{x})}^{|\alpha| - |\gamma|}} \leq \frac{c}{\kappa_\epsilon^{|\alpha|-|\gamma|}}.
\end{equation}
Therefore 
\begin{equation}\label{estimatederh}
\norm{D^\alpha h_{\epsilon,j}}_\infty \leq c \sum_{0 \leq \gamma \leq \alpha} \frac{\norm{D^\gamma (g_{\epsilon,j} - g_j)}_\infty}{\kappa_\epsilon^{|\alpha|- |\gamma|}}
\end{equation}
for all $\epsilon>0$ sufficiently small. It follows by the definitions of $\Psi_{\epsilon,j}, \Phi_{\epsilon,j}$, by  \eqref{estimatederh} and part $(iii)$ of condition \eqref{assumptions},  that for all $m,i =1,2,3$
\begin{equation} \label{secondderpsi}
\norm{\frac{\partial^2 \Psi^m_{\epsilon,j}}{\partial x_i^2}}_{L^\infty(\Omega_\epsilon \cap V_j)} = o(\kappa_\epsilon^{-1/2}),\ \   {\rm as}\  \epsilon \to 0. 
\end{equation}

We claim that 
\begin{equation} \label{divtildephigoesto0}
\lim_{\epsilon \to 0} \int_{(\Omega_\epsilon \cap V_j) \setminus r^{-1}_j(K_{\epsilon,j})} \abs{\operatorname{div}\tilde{\varphi}_{\epsilon,j}}^2 =0.
\end{equation}
To prove that, we analyse first the terms of type A in \eqref{type:A:and:B}. By changing variables in integrals we get:
\begin{equation}  \label{estimateA}
\begin{split}
&\int_{(\Omega_\epsilon \cap V_j) \setminus r^{-1}_j(K_{\epsilon,j})} \abs{\frac{\partial \varphi_j^m}{\partial x_n} \circ \Psi_{\epsilon,j}}^2 \abs{\frac{\partial \Psi_{\epsilon,j}^n}{\partial x_i}}^2 \abs{\frac{\partial \Psi_{\epsilon,j}^m}{\partial x_i}}^2 dy\\
&\leq c \int_{(\Omega_\epsilon \cap V_j) \setminus r^{-1}_j(K_{\epsilon,j})} \abs{\frac{\partial \varphi_j^m}{\partial x_n} \circ \Psi_{\epsilon,j}}^2 dy\\
&=c \int_{(\Omega \cap V_j) \setminus r^{-1}_j(K_{\epsilon,j})} \abs{\frac{\partial \varphi_j^m}{\partial x_n}}^2 \frac{1}{\abs{\operatorname{det}(\operatorname{D}\Psi_{\epsilon,j}) \circ \Psi^{(-1)}_{\epsilon,j}}} \, dx\\
&\leq c \int_{(\Omega \cap V_j) \setminus r^{-1}_j(K_{\epsilon,j})} \abs{\frac{\partial \varphi_j^m}{\partial x_n}}^2 dx \xrightarrow[\epsilon \to 0]{} 0.
\end{split}
\end{equation}
We now consider  the terms of type B. By setting  $\eta_j(z) := \varphi_j(r^{-1}_j(z))$ and recalling \eqref{secondderpsi} we have that
\begin{equation}\label{estimateAbis}
\begin{split}
&\int_{(\Omega_\epsilon \cap V_j) \setminus r^{-1}_j(K_{\epsilon,j})} \abs{\varphi_j^m (\Psi_{\epsilon,j}) \frac{\partial^2 \Psi^m_{\epsilon,j}}{\partial x_i^2}}^2 dy \\
& \qquad \leq \norm{\frac{\partial^2 \Psi^m_{\epsilon,j}}{\partial x_i^2}}^2_{L^\infty(\Omega_\epsilon \cap V_j)} \int_{(\Omega_\epsilon \cap V_j) \setminus r^{-1}_j(K_{\epsilon,j})} \abs{\varphi_j (\Psi_{\epsilon,j})}^2 dy\\
& \qquad = o(\kappa_\epsilon^{-1}) \int_{(\Omega \cap V_j) \setminus r^{-1}_j(K_{\epsilon,j})} \abs{\varphi_j}^2 \frac{1}{\abs{\operatorname{det}(\operatorname{D}\Psi_{\epsilon,j}) \circ \Psi^{(-1)}_{\epsilon,j}}} \, dx\\
& \qquad \leq o(\kappa_\epsilon^{-1}) \int_{(\Omega \cap V_j) \setminus r^{-1}_j(K_{\epsilon,j})} \abs{\varphi_j(x)}^2 dx \\
&\qquad = o(\kappa_\epsilon^{-1}) \int_{r_j(\Omega \cap V_j) \setminus K_{\epsilon,j}} \abs{\eta_j (z)}^2 dz \\
& \qquad = o(\kappa_\epsilon^{-1}) \int_{W_j} \left( \int_{{\hat g}_{\epsilon,j}(\bar{z})}^{g_j(\bar{z})} \abs{\eta_j(\bar{z},z_3)}^2 dz_3 \right) d\bar{z} \\
&\qquad \leq o(\kappa_\epsilon^{-1}) \int_{W_j} \abs{g_j(\bar{z}) - {\hat g}_{\epsilon,j}(\bar{z})} \norm{\eta_j(\bar{z}, \cdot)}^2_{L^\infty(a_{3j}, g_j(\bar{z}))^3} d\bar{z} \\
&\qquad \leq o(\kappa_\epsilon^{-1}) \norm{g_j - {\hat g}_{\epsilon,j}}_{L^\infty(W_j)} \int_{W_j} \norm{\eta_j(\bar{z}, \cdot)}^2_{H^1(a_{3j}, g_j(\bar{z}))^3} d\bar{z} \\
&\qquad \leq o(\kappa_\epsilon^{-1}) \, \kappa_\epsilon \norm{\eta_j}_{H^1(r_j(\Omega \cap V_j))^3}^2 \xrightarrow[\epsilon \to 0]{} 0.
\end{split}
\end{equation}
Here we have used the following one dimensional embedding estimate for Sobolev functions (see e.g., Burenkov \cite{bur}):
\begin{equation*}
\norm{f}_{L^\infty(a,b)} \leq c \norm{f}_{H^1(a,b)}
\end{equation*}
for all $f \in H^1(a,b)$, where the constant $c=c(d)$ is uniformly bounded for $\abs{b-a}>d$. We conclude that  \eqref{divtildephigoesto0} holds. 

Using \eqref{divtildephigoesto0}, the fact that $\Psi_{\epsilon,j}$ in $r^{-1}_j(K_{\epsilon,j})$ coincides with the identity and that \\
$\abs{(\Omega \cap V_j) \setminus r^{-1}_j(K_{\epsilon,j})} \to 0$ as $\epsilon$ goes to 0, we deduce that
\begin{equation} \label{divtildephionKepsilon}
\lim_{\epsilon \to 0} \int_{r^{-1}_j(K_{\epsilon,j})} \abs{\operatorname{div}\tilde{\varphi}_{\epsilon,j}}^2 = \int_\Omega \abs{\operatorname{div}\varphi_j}^2,
\end{equation}
and
\begin{equation} \label{divtildephigoestophi}
\lim_{\epsilon \to 0} \int_{\Omega_\epsilon} \abs{\operatorname{div}\tilde{\varphi}_{\epsilon,j}}^2 = \int_\Omega \abs{\operatorname{div}\varphi_j}^2.
\end{equation}

With \eqref{divtildephigoesto0}, \eqref{divtildephionKepsilon} and \eqref{divtildephigoestophi} in mind, in order to prove \eqref{divintegrale1}, it suffices to  reproduce the same argument used before   starting from \eqref{dottildephisplitintegral} combined with formula \eqref{setidentity}. We omit the details.  
Thus statement (iii) is proved. 

The proof of statement (iv) follows by the same considerations above. First of all, for any $j=1, \dots, s'$ the function $\tilde \varphi_{\epsilon, j}$ coincides with $\varphi_j$ 
on $r_j^{-1}(K_{\epsilon , j})$. Thus ${\mathcal P}_{\epsilon}\varphi =\varphi $ on 
 $(\cup_{j=1,\dots , s'}  r_j^{-1}(K_{\epsilon , j})  ) \cup  ( \cup_{j=s'+1, \dots , s} V_j)$. 
 It follows that 
 $$\|   {\mathcal P}_{\epsilon}\varphi - \varphi  \|_{X(\Omega_{\epsilon}\cap \Omega  )}  \le \|   {\mathcal P}_{\epsilon}\varphi - \varphi  \|_{X( \cup_{j=1,\dots , s'} (\Omega_{\epsilon}\cap V_j)   \setminus r_j^{-1}(K_{\epsilon , j}  )  )}
$$
This combined with  by the limiting relations \eqref{tildephigoesto0}, \eqref{curltildephigoesto0} and \eqref{divtildephigoesto0}  yields the validity of statement (iv). 
\end{proof}

\section{Spectral stability}
\label{sezione4}

Let $\Omega_\epsilon$,   $\epsilon >0$,  and $\Omega$ be bounded domains in $\mathbb{R}^3$ of class $C^{1,1}(\mathcal{A})$. For simplicity, it is convenient to  set $\Omega_0=\Omega$.  In this section, we prove that if  
 $\Omega_{\epsilon}$ converges to $\Omega$ as $\epsilon\to 0$ in the sense of  \eqref{assumptions}, and  a uniform Gaffney inequality holds on the domains $\Omega_{\epsilon}$ then we have spectral stability for the $\cu\cu $ operator defined on the domains $\Omega_{\epsilon}$ with respect to the reference domain $\Omega$.
By uniform Gaffney inequality, we mean that the spaces $\xn(\Omega_{\epsilon})$ are embedded into $H^1(\Omega_{\epsilon})^3$ and there exists a positive constant $C$ independent of $\epsilon$ such that 
 \begin{equation}\label{unigaff}
 \| u\|_{H^1(\Omega_{\epsilon})^3}\le C \| u\|_{\xn(\Omega_{\epsilon})}\, , 
 \end{equation}
 for all $u\in \xn(\Omega_{\epsilon})$ and $\epsilon >0$. (Note that by Theorem~\ref{gaffneythm}, for every $\epsilon>0$ there exists a positive constant $C_{\epsilon}$, possibly depending  on $\epsilon$, such that \eqref{unigaff} holds, but here we need a constant independent of $\epsilon$).

To do so,  for any $\epsilon \geq 0$, we denote by $S_{\epsilon}$ the operator $S_{\Omega_{\epsilon}}$ from $L^2(\Omega_{\epsilon})$ to itself defined in Lemma~\ref{esse}.
Recall that  if $f_\epsilon\in L^2(\Omega_\epsilon)^3$ is the datum of the following Poisson problem
\begin{equation}\label{Poissonprob}
 \left\{
        \begin{array}{ll}
            \operatorname{curl}\operatorname{curl} v_\epsilon -\tau \nabla \operatorname{div}v_\epsilon + v_\epsilon = f_\epsilon, & \quad \text{in } \Omega_\epsilon,\\
            \operatorname{div}v_\epsilon=0,& \quad \text{on } \partial \Omega_\epsilon,\\
            v_\epsilon \times \nu = 0, & \quad \text{on } \partial\Omega_\epsilon,
        \end{array}
    \right.
\end{equation}
then the unique solution $  v_{\epsilon }\in \xn(\Omega_{\epsilon })$ is precisely $S_{\epsilon}f_{\epsilon }$, that is $v_{\epsilon }=S_{\epsilon}f_{\epsilon }$.
Recall that $\tau $ is a fixed positive constant (one could normalise it by setting $\tau =1$ but we prefer to keep it as it is also with reference to other papers 
where it is important to have the possibility to use different values of $\tau$, see for example  \cite[Remark~2.13]{lamzac}).
In this section we  prove that $S_{\epsilon}$ compactly converges to $S_{0}$ as $\epsilon \to 0$, and this  implies spectra stability.     This has to be understood in the following sense.

We denote by $E=\{ E_\epsilon \}_{\epsilon >0} $ the family of the extension-by-zero operators $E_\epsilon:L^2(\Omega)^3 \to L^2(\Omega_\epsilon)^3$ defined by
\begin{equation}\label{def:extby0}
E_\epsilon \varphi =\varphi^0= \left\{
        \begin{array}{ll}
            \varphi, & \quad \text{if } x\in \Omega_\epsilon \cap \Omega,\vspace{1mm}\\
            0, & \quad \text{if } x \in \Omega_\epsilon \setminus \Omega  , 
        \end{array}
    \right.
\end{equation}
for all $\varphi \in L^2(\Omega)^3$.
Note that under our assumptions we have that  for all $\varphi \in L^2(\Omega)^3$
$$\lim_{\epsilon \to 0} \norm{E_\epsilon \varphi}_{L^2(\Omega_\epsilon)^3} = \norm{\varphi}_{L^2(\Omega)^3},$$
since
$\abs{\Omega \setminus (\Omega_\epsilon \cap \Omega)} \to 0$ as $\epsilon$ goes to 0.    We recall the following definition from \cite{Vai}, see also \cite{arcalo} and \cite{capi}. 

\begin{defn} Let $u_{\epsilon}\in  L^2(\Omega_{\epsilon })$, for $\epsilon >0$, be a family of functions. We say that $u_{\epsilon}$ $E$-converges to $u_0\in L^2(\Omega)$ as $\epsilon \to 0$ and we write $u_{\epsilon}\xrightarrow{E} u_0$  if 
$$
\| u_{\epsilon} -E_{\epsilon }u_0\|_{L^2(\Omega_{\epsilon })}\to 0,\ \ {\rm as}\ \ \epsilon \to 0\, .
$$
We also say that $S_{\epsilon }$   $E$-converges to $S_0$ as $\epsilon \to 0$ and we write $S_{\epsilon}\xrightarrow{EE} S_0$  if 
for any family of functions   $f_{\epsilon }\in L^2(\Omega_{\epsilon })$,  we have
$$
f_{\epsilon } \xrightarrow{E} f_0\ \    \Longrightarrow \ \     S_{\epsilon }f_{\epsilon } \xrightarrow{E} S_0f_0\, .
$$

Finally, we say that  $S_{\epsilon }$   $E$-compact converges  to $S_0$ as $\epsilon \to 0$ and we write $S_{\epsilon}\xrightarrow{C} S_0$ if $S_{\epsilon}\xrightarrow{EE} S_0$ and for any family of functions $f_{\epsilon }\in L^2(\Omega_{\epsilon })$, with $\| f_{\epsilon } \|_{L^2(\Omega_{\epsilon})}=1$ and any sequence of positive numbers $\epsilon_n$
with $\epsilon_n\to 0$, there exists a subsequence $\epsilon_{n_k}$ and $u\in L^2(\Omega)$ such that $  S_{\epsilon_{n_k}}f_{n_k}   \xrightarrow{E} u   $. 
 \end{defn}

The following theorem from \cite[Thm.~6.3]{Vai} holds.

\begin{theorem}\label{vainikko} If 
 $S_{\epsilon}\xrightarrow{C} S_0$ then the eigenvalues of the operator $S_{\epsilon}$ converge to the eigenvalues of the operator $S_0$, and   the eigenfunctions of the operator $S_{\epsilon}$ $E$-converge to the eigenfunctions of the operator $S_{0}$ as $\epsilon \to 0$. 
\end{theorem}

If we denote by $\mu_n(\epsilon)$, $n\in \mathbb{N}$ the sequence of eigenvalues of $S_{\epsilon}$ and by $u_n(\epsilon)$,  $n\in \mathbb{N}$ a corresponding orthonormal sequence of eigenfunctions, then the stability of eigenvalues and eigenfunctions stated above has to be interpreted in the following sense:

\begin{itemize}
 \item[(i)]       $\mu_n(\epsilon )\to \mu_n(0)$ as $\epsilon \to 0$.
 \item[(ii)] For any sequence $\epsilon_k$, $k\in  \mathbb{N}$,  converging to zero there exists an orthonormal sequence of eigenfunctions  $u_n(0)$, $n\in  \mathbb{N}$ in $L^2(\Omega)^3$ such that, possibly passing to a subsequence of $\epsilon_k$,   $u_n(\epsilon_k)   \xrightarrow{E}    u_n(0)$.
 \item[(iii)]  Given  $m$ eigenvalues $\mu_n(0),  \dots , \mu_{n+m-1}(0)$ with
$\mu_n(0)\ne \mu_{n-1}(0)$ and $\mu_{n+m-1}(0)\ne \mu_{n+m}(0)$
 and corresponding orthonormal eigenfunctions $u_n(0),  \dots , u_{n+m-1}(0)$,
 there exist $m$ orthonormal   generalized eigenfunctions (i.e. linear combinations of eigenfunctions)   $v_n(\epsilon ),  \dots , v_{n+m-1}(\epsilon )$  associated with  $\mu_n(\epsilon ),  \dots ,
  \mu_{n+m-1}(\epsilon )$   such that $v_{n+i}(\epsilon ) \xrightarrow{E} u_{n+i}(0)$  for all $i=0, 1,\dots , m-1$.
 \end{itemize}

Recall  that $\mu $ is an eigenvalue of $S_\epsilon$ if and only if $\lambda=\mu^{-1}$ is an eigenvalue of the problem
\begin{equation}\label{eigen}
\begin{cases}
\operatorname{curl}\operatorname{curl} v_\epsilon -\tau \nabla \operatorname{div}v_\epsilon +v_\epsilon =\lambda v_\epsilon, & \text{in }\Omega_\epsilon,\\
\operatorname{div}v_\epsilon=0, & \text{on }\partial\Omega_\epsilon,\\
v_\epsilon \times \nu =0, & \text{on }\partial\Omega_\epsilon.
\end{cases}
\end{equation}
and that the corresponding eigenfunctions are the same. Note that  the eigenvalues of \eqref{eigen} differ from those of \eqref{mainpen} just by a translation. Thus, studying the stability of eigenvalues and eigenfunctions of the problem \eqref{eigen} or \eqref{mainpen},  is equivalent to studying the spectral stability of the family of operators $S_{\epsilon}$.  To do so, we recall that the weak formulation of problem \eqref{Poissonprob}
reads as follows: find  $v_\epsilon \in \xn(\Omega_\epsilon)$ such that
\begin{equation}\label{Poissonprobweak}
\int_{\Omega_\epsilon} v_\epsilon \cdot \eta \, dx + \int_{\Omega_\epsilon} \operatorname{curl}v_\epsilon \cdot \operatorname{curl}\eta \, dx + \tau \int_{\Omega_\epsilon} \operatorname{div}v_\epsilon \operatorname{div}\eta \, dx = \int_{\Omega_\epsilon} f_\epsilon \cdot \eta \, dx
\end{equation}
for all $\eta \in \xn(\Omega_\epsilon)$.

Suppose that for every $\epsilon>0$ we have that $\norm{f_\epsilon}_{L^2(\Omega_\epsilon)^3} \leq C$ for some $C>0$. Then, setting  $\eta = v_\epsilon$ in  \eqref{Poissonprobweak} and observing that $\int_{\Omega_\epsilon} f_\epsilon \cdot v_\epsilon \, dx \leq \frac{1}{2} \int_{\Omega_\epsilon}\abs{f_\epsilon}^2 \, dx + \frac{1}{2} \int_{\Omega_\epsilon} \abs{v_\epsilon}^2 \, dx$, we get 
$$\frac{1}{2}\int_{\Omega_\epsilon} \abs{v_\epsilon}^2 \, dx + \int_{\Omega_\epsilon} \abs{\operatorname{curl}v_\epsilon}^2 \, dx + \tau \int_{\Omega_\epsilon} \abs{\operatorname{div}v_\epsilon}^2 \, dx \leq \frac{1}{2} \int_{\Omega_\epsilon} \abs{f_\epsilon}^2 \, dx.$$
This in turn implies that for all $\epsilon>0$
\begin{equation} \label{normveps}
\begin{split}
\norm{v_\epsilon}_{\xn(\Omega_\epsilon)} &= \left( \norm{v_\epsilon}^2_{L^2(\Omega_\epsilon)^3} + \,\norm{\operatorname{curl} v_\epsilon}^2_{L^2(\Omega_\epsilon)^3} + \,\norm{\operatorname{div}v_\epsilon}^2_{L^2(\Omega_\epsilon)} \right)^{1/2}
\leq c \norm{f_\epsilon}_{L^2(\Omega_\epsilon)} =O(1). 
\end{split}
\end{equation}

In order to prove the $E$-convergence of the operators $S_{\epsilon}$, it is necessary to consider the limit of functions $v_{\epsilon}$. We note that if $\Omega \subset \Omega_{\epsilon}$ for all $\epsilon >0$ then it would  suffice to consider the restriction of $v_{\epsilon}$ to $\Omega$ and pass to the weak limit in $\Omega$. Otherwise, it is convenient to extend functions $v_{\epsilon }$ to the whole of ${\mathbb{R}}^3$. To do so, we observe that by   the uniform Gaffney inequality combined with 
inequality \eqref{normveps} it follows that $\| v_{\epsilon }\|_{H^1(\Omega_{\epsilon})^3}$ is uniformly bounded. Moreover, the domains $\Omega_{\epsilon}$ belong to the same Lipschitz class $C^{0,1}_M({\mathcal{A}})$ for some $M>0$ hence the functions $v_{\epsilon }$ can  be extended to the whole of ${\mathbb{R}}^3$
with a uniformly bounded norm, see e.g., \cite{bur}. Thus, in the sequel we shall directly make the following assumption: 
\begin{equation}\label{globalass}
v_{\epsilon }\in H^1(\mathbb{R}^3)^3\cap \xn(\Omega_{\epsilon}),\ \ \sup_{\epsilon>0}\| v_{\epsilon }\|_{\in H^1(\mathbb{R}^3)^3}\ne \infty\, .
\end{equation}
Thus the family  $\{v_\epsilon\rvert_{ \Omega }  \}_{\epsilon>0}$ is bounded in $H(\operatorname{curl}; \Omega) \cap H(\operatorname{div}; \Omega)$ and  we can extract a sequence $\{v_{\epsilon_n}\rvert_\Omega\}_{n \in \mathbb{N}}$, with $\epsilon_n \to 0$ as $n$ goes to $\infty$, such that
\begin{equation} \label{def:weaklimit:v}
v_{\epsilon_n}\rvert_\Omega \underset{n \to \infty}{\rightharpoonup}  v  \qquad \text{weakly in } H(\operatorname{curl}; \Omega) \cap H(\operatorname{div}; \Omega)
\end{equation}
for some $v \in H(\operatorname{curl}; \Omega) \cap H(\operatorname{div}; \Omega)$. It turns out that $v$ preserves the boundary conditions as the following lemma clarifies.

\begin{lemma} \label{v:is:XN} Assume  that for some $\epsilon_n>0$ with $\epsilon_n\to 0$, there exists $v\in H(\operatorname{curl},\Omega) \cap H(\operatorname{div},\Omega)$ such that   $\{v_{\epsilon_n}\rvert_\Omega\}_{n \in \mathbb{N}}$ weakly converges to $v$ in $H(\operatorname{curl},\Omega) \cap H(\operatorname{div},\Omega)$. Then $v \in \xn(\Omega)$.
\end{lemma}
\begin{proof}
To prove that $v \in \xn(\Omega)$ we just need to make sure  that $v\in H_0(\operatorname{curl},\Omega)$.
Since $v_{\epsilon_n} \in H_0(\operatorname{curl},\Omega_\epsilon)$ for all $n \in \mathbb{N}$, by \Cref{extbyzero} we know that the extension-by-zero $v_{\epsilon_n}^0$ of $v_{\epsilon_n}$ belongs to $H(\operatorname{curl},\mathbb{R}^3)$ for all $n \in \mathbb{N}$. By the reflexivity of $H(\operatorname{curl},\mathbb{R}^3)$ and the boundedness of the  sequence $\set{v^0_{\epsilon_n}}_{n \in \mathbb{N}}$, we deduce that possibly passing to  a subsequence, there exists a function $\tilde{v} \in H(\operatorname{curl}, \mathbb{R}^3)$ such that $v^0_{\epsilon_n} \rightharpoonup \tilde{v}$ weakly in $H(\operatorname{curl}, \mathbb{R}^3)$ as $n$ goes to $\infty$. It suffices to show that $\tilde{v}=v^0$.  Since $v^0_{\epsilon_n}$ is equal to zero outside of $\Omega_{\epsilon_n}$, it is clear  that $\tilde{v}=0$ a.e. in $\mathbb{R}^3 \setminus \Omega$. Moreover, since  $v_{\epsilon_n}\rvert_\Omega$ weakly converges to both $v, \tilde v$ in $H(\operatorname{curl}, \Omega )$, we have that $v=\tilde v$  a.e. in $\Omega$. Thus the extension 
by zero of $v$ to the whole of $\mathbb{R}^3$ is precisely $\tilde v$ and belongs to $ H(\operatorname{curl},{\mathbb{R}^3}) $. Using \Cref{extbyzero} again, we see that $v \in H_0(\operatorname{curl},\Omega)$. 
\end{proof}

\begin{lemma} \label{limitisthesolution}
 Assume that  condition \eqref{assumptions} and  the uniform Gaffney inequality \eqref{unigaff} hold. 
For any $\epsilon >0$  let $f_\epsilon \in L^2(\Omega_\epsilon)^3$.  Suppose that $\sup_{\epsilon > 0}\norm{f_\epsilon}_{L^2(\Omega_\epsilon)^3} \ne \infty$ and that the extension-by-zero of the functions $f_\epsilon$ converge weakly in $L^2(\Omega)^3$ to some function $f \in L^2(\Omega)^3$  as $\epsilon \to 0$. For all $\epsilon >0$, let  $v_\epsilon := S_\epsilon f_\epsilon$ the (unique) weak solution in $\xn(\Omega_\epsilon)$ of \eqref{Poissonprobweak} with datum $f_\epsilon$. Assume \eqref{globalass} and suppose that $v_\epsilon \rightharpoonup v$ weakly in $H(\operatorname{curl}, \Omega) \cap H(\operatorname{div}, \Omega)$ to some $v\in  H(\operatorname{curl}, \Omega) \cap H(\operatorname{div}, \Omega)$. Then $v = S_0 f$.
\end{lemma}
\begin{proof}
First of all, we note that by \Cref{v:is:XN}, $v\in \xn(\Omega)$. 
Define for $u,w \in H(\operatorname{curl},\Omega_\epsilon) \cap H(\operatorname{div},\Omega_\epsilon)$
\begin{equation*}
Q_{\Omega_\epsilon} (u,w):= \int_{\Omega_\epsilon} u \cdot w \, dx + \int_{\Omega_\epsilon} \operatorname{curl}u \cdot \operatorname{curl}w \, dx + \tau \int_{\Omega_\epsilon} \operatorname{div}u \cdot \operatorname{div}w \, dx,
\end{equation*}
which is equivalent to  the scalar product for the space $H(\operatorname{curl},\Omega_\epsilon) \cap H(\operatorname{div},\Omega_\epsilon)$. The square of the induced norm will be denoted with $Q_{\Omega_\epsilon}(\cdot)$.
Note that since $v_\epsilon$ is the solution with datum $f_\epsilon$, then we have that 
$$Q_{\Omega_\epsilon}(v_\epsilon, \eta) = \int_{\Omega_\epsilon} f_\epsilon \cdot \eta$$
for all $\eta \in \xn(\Omega_\epsilon)$.

Let $\varphi$ be any function in $\xn(\Omega)$ and let ${\mathcal{P}}_{\epsilon }\varphi $ the Atlas Piola trasform of $\varphi$. Since 
${\mathcal{P}}_{\epsilon }\varphi \in \xn(\Omega_{\epsilon})$, we deduce tha
\begin{equation}
\label{eqpiola}
Q_{\Omega_\epsilon}(v_\epsilon, {\mathcal{P}}_{\epsilon }\varphi )= \int_{\Omega_\epsilon}     f_\epsilon \cdot  {\mathcal{P}}_{\epsilon }\varphi 
\end{equation}
for all $\epsilon>0$.  We now show that
\begin{equation} \label{limitfeps}
\lim_{\epsilon \to 0} \int_{\Omega_\epsilon}     f_\epsilon \cdot  {\mathcal{P}}_{\epsilon }\varphi = \int_\Omega f \cdot \varphi
\end{equation}
and
\begin{equation} \label{limitQeps}
\lim_{\epsilon \to 0} Q_{\Omega_\epsilon}(v_\epsilon, {\mathcal{P}}_{\epsilon }\varphi ) = Q_\Omega(v, \varphi)
\end{equation}
In order to prove the first limit, it suffices to  prove that 
\begin{equation} \label{fepsj'}
\int_{\Omega_\epsilon \cap V_j} f_\epsilon \cdot \tilde{\varphi}_{\epsilon,j} \xrightarrow[\epsilon \to 0]{} \int_{\Omega\cap V_j} f \cdot \varphi_j
\end{equation}
for any $j=1,\dots,s'$,   where  $\tilde{\varphi}_{\epsilon,j} $ is defined in \eqref{tildefij} (with $\tilde \Omega $ replaced by $\Omega_{\epsilon}$),    since it is obvious that 
\begin{equation} \label{fepsj}
\int_{\Omega_\epsilon \cap V_j} f_\epsilon \cdot \varphi_j \xrightarrow[\epsilon \to 0]{} \int_{\Omega\cap V_j} f \cdot \varphi_j
\end{equation}
for any $j=s'+1,\dots,s$. 
We have that 
\begin{equation} \label{splitintegralfeps}
\int_{\Omega_\epsilon \cap V_j} f_\epsilon \cdot \tilde{\varphi}_{\epsilon,j} = \int_{r_j^{-1}(K_{\epsilon,j})} f_\epsilon \cdot \varphi_j  + \int_{(\Omega_\epsilon \cap V_j) \setminus r_j^{-1}(K_{\epsilon,j})} f_\epsilon \cdot \tilde{\varphi}_{\epsilon,j}.
\end{equation}
Obviously
\begin{equation} \label{fepsonKeps}
\lim_{\epsilon \to 0} \int_{r_j^{-1}(K_{\epsilon,j})} f_\epsilon \cdot \varphi_j = \int_{\Omega \cap V_j} f \cdot \varphi_j
\end{equation}
since the extension-by-zero of $f_\epsilon$ weakly converge to $f$ in $L^2(\Omega)^3$, $\sup_{\epsilon >0} \norm{f_\epsilon}_{L^2(\Omega_\epsilon)^3} < \infty$ and $\abs{(\Omega \cap V_j) \setminus r_j^{-1}(K_{\epsilon,j})}$ goes to 0 as $\epsilon \to 0$.
Meanwhile
\begin{equation} \label{fepsgoesto0}
\begin{split}
&\left|\int_{(\Omega_\epsilon \cap V_j) \setminus r_j^{-1}(K_{\epsilon,j})} f_\epsilon \cdot \tilde{\varphi}_{\epsilon,j} \right|\leq \norm{f_\epsilon}_{L^2(\Omega_\epsilon)^3} \left(\int_{(\Omega_\epsilon \cap V_j) \setminus r_j^{-1}(K_{\epsilon,j})} \abs{\tilde{\varphi}_{\epsilon,j}}^2 \right)^\frac{1}{2} \xrightarrow[\epsilon \to 0]{}0
\end{split}
\end{equation}
by \eqref{tildephigoesto0} and the hypothesis that $\sup_{\epsilon > 0}\norm{f_\epsilon}_{L^2(\Omega_\epsilon)^3} \ne \infty$.
From \eqref{splitintegralfeps}, \eqref{fepsonKeps} and \eqref{fepsgoesto0} we immediately deduce \eqref{fepsj'}.
 Hence we have proved \eqref{limitfeps}.

Let us now focus on \eqref{limitQeps}.  We write 
\begin{eqnarray}\lefteqn{
 Q_{\Omega_\epsilon}(v_\epsilon, {\mathcal{P}}_{\epsilon }\varphi  ) =  Q_{\Omega_\epsilon \cap \Omega }(v_\epsilon, {\mathcal{P}}_{\epsilon }\varphi  )+
  Q_{\Omega_\epsilon \setminus \Omega }(v_\epsilon, {\mathcal{P}}_{\epsilon }\varphi  )  }\nonumber \\
  & &= Q_{\Omega_\epsilon \cap \Omega }(v_\epsilon, {\mathcal{P}}_{\epsilon }\varphi  -\varphi )+
  Q_{\Omega_\epsilon \cap \Omega }(v_\epsilon, \varphi ) +
  Q_{\Omega_\epsilon \setminus \Omega }(v_\epsilon, {\mathcal{P}}_{\epsilon }\varphi  )\nonumber  \\
  & & = Q_{\Omega }(v_\epsilon, \varphi ) -  Q_{\Omega \setminus \Omega_{\epsilon }   }(v_\epsilon, \varphi ) 
  +  Q_{\Omega_\epsilon \cap \Omega }(v_\epsilon, {\mathcal{P}}_{\epsilon }\varphi  -\varphi )  +
    Q_{\Omega_\epsilon \setminus \Omega }(v_\epsilon, {\mathcal{P}}_{\epsilon }\varphi  )
\end{eqnarray}

By the weak convergence of $v_{\epsilon }$ to $v$ we have that 
\begin{equation}\label{limitisthesolution1}
 Q_{\Omega }(v_\epsilon, \varphi ) \to  Q_{\Omega }(v, \varphi ), \ \ {\rm as}\ \epsilon\to 0 .
\end{equation}

By the Cauchy-Schwarz inequality and Theorem~\ref{Piolamain}, (iv) we get that 
\begin{equation}\label{limitisthesolution2}
 Q_{\Omega_\epsilon \cap \Omega }(v_\epsilon, {\mathcal{P}}_{\epsilon }\varphi  -\varphi ) \le ( Q_{\Omega_\epsilon \cap \Omega }(v_\epsilon )  )^{1/2}(Q_{\Omega_\epsilon \cap \Omega }( {\mathcal{P}}_{\epsilon }\varphi  -\varphi ) )^{1/2}\to 0,\ \ {\rm as}\ \epsilon \to 0
\end{equation}

Similarly, 
\begin{equation}\label{limitisthesolution3}
 Q_{\Omega \setminus \Omega_{\epsilon }   }(v_\epsilon, \varphi ) \le  (Q_{\Omega \setminus \Omega_{\epsilon }  }(v_\epsilon)  )^{1/2}
  (Q_{\Omega \setminus \Omega_{\epsilon }   }( \varphi ) )^{1/2}\to 0 \ \ {\rm as}\ \epsilon\to 0 .
\end{equation}

Finally, 
\begin{equation}\label{limitisthesolution4}
 Q_{\Omega_\epsilon \setminus \Omega }(v_\epsilon, {\mathcal{P}}_{\epsilon }\varphi  )\le 
( Q_{\Omega_\epsilon \setminus \Omega }(v_\epsilon))^{1/2}
 (Q_{\Omega_\epsilon \setminus \Omega }({\mathcal{P}}_{\epsilon }\varphi  ))^{1/2} \to 0  \ \ {\rm as}\ \epsilon\to 0 
\end{equation}
since by \eqref{tildephigoesto0}, \eqref{curltildephigoesto0} and \eqref{divtildephigoesto0}  it follows that $Q_{\Omega_\epsilon \setminus \Omega }({\mathcal{P}}_{\epsilon }\varphi  )\to 0$ as $\epsilon \to 0$.
By combining  \eqref{limitisthesolution1}-\eqref{limitisthesolution4}, we deduce that limit \eqref{limitQeps} holds.

In conclusion, by using the limiting relations  \eqref{limitfeps} and \eqref{limitQeps} in  equation \eqref{eqpiola} we conclude that 
$$
Q_{\Omega}(v,\varphi )= \int_{\Omega }f\cdot \varphi
$$
which means that 
$v$ is the solution in $\xn(\Omega)$ of the given problem  with datum $f\in L^2(\Omega)$, as required. 

\end{proof}

\begin{remark} A careful inspection of the proof of Lemma~\ref{limitisthesolution} reveals that the uniform Gaffney inequality has been used only to prove the limiting relations \eqref{limitisthesolution1}-\eqref{limitisthesolution3} since the functions $v_{\epsilon}$ are required here to be defined on $\Omega$ and to have uniformly bounded norms. This problem does not occur if 
$\Omega \subset \Omega_{\epsilon}$ in which case only the Gaffney inequality in $\Omega$ is necessary. However, the uniform Gaffney inequality will be 
used in an essential way in the following statements also in the particular case $\Omega \subset \Omega_{\epsilon}$
\end{remark}

In the next lemma we prove that  $S_\epsilon $ $E$-converges to $S_0$ as $\epsilon \to 0$.

\begin{lemma}\label{lemma:EE:convergence}
 Assume that  condition \eqref{assumptions} and  the uniform Gaffney inequality \eqref{unigaff} hold. 
 Let $f_\epsilon \in L^2(\Omega_\epsilon)^3, \epsilon >0$ be such that $f_\epsilon \xrightarrow[\epsilon \to 0]{E} f \in L^2(\Omega)^3$  for some function $f \in L^2(\Omega)^3$. Set $v_\epsilon := S_\epsilon f_\epsilon$ and $v:=S_0 f$. Then $v_\epsilon \xrightarrow[\epsilon \to 0]{E} v$, hence $S_\epsilon \xrightarrow[\epsilon \to 0]{EE}S_0.$
\end{lemma}
\begin{proof}
Since  $f_\epsilon \xrightarrow[\epsilon \to 0]{E} f$, then  $\norm{f_\epsilon}_{L^2(\Omega_\epsilon)^3} \leq C$  for all $\epsilon>0$  sufficiently small and consequently $\norm{v_\epsilon}_{\xn(\Omega_\epsilon)}$ is uniformly bounded with respect to $\epsilon$,  as shown in \eqref{normveps}.
By the uniform Gaffney inequality it follows that  also  $\norm{v_\epsilon}_{H^1(\Omega_\epsilon)^3} $ is uniformly bounded. 
 In particular
$$\lim_{\epsilon \to 0} \norm{v_\epsilon}_{L^2(\Omega_\epsilon \setminus \Omega)^3} =0$$
because $\abs{\Omega_\epsilon \setminus \Omega}\to 0$ as $\epsilon$ goes to 0. This can be proved using the same argument used for \eqref{estimateAbis}  as follows:
\begin{equation*}
\begin{split}
& \quad  \int_{(\Omega_{\epsilon} \cap V_j) \setminus r^{-1}_j(K_{\epsilon,j})} \abs{v_{\epsilon}(x)   }^2 dx 
=  \int_{r_j(\Omega_{\epsilon} \cap V_j) \setminus K_{\epsilon,j}} \abs{  v_{\epsilon}\circ r_j^{-1} (z)    }^2 dz \\
& \quad =  \int_{W_j} \left( \int_   {{\hat g}_{\epsilon,j}(\bar{z})}    ^{{g}_{\epsilon,j}(\bar{z})}        \abs{   v_{\epsilon}\circ r_j^{-1}(\bar{z},z_3)     }^2 dz_3 \right) d\bar{z} \\
&\quad \leq  \int_{W_j} \abs    {   {g}_{\epsilon,j}(\bar{z}) - {\hat g}_{\epsilon,j}(\bar{z})      } \norm{   v_{\epsilon}\circ r_j^{-1}   (   \bar{z}, \cdot)}^2_{L^\infty(a_{3j}, g_{\epsilon , j}(\bar{z}))^3} d\bar{z} \\
&\quad \leq \norm{g_{\epsilon ,j }- {\hat g}_{\epsilon,j}}_{L^\infty(W_j)} \int_{W_j} \norm{  v_{\epsilon}\circ r_j^{-1}    (\bar{z}, \cdot)}^2_{H^1(a_{3j}, g_{\epsilon , j}(\bar{z}))^3} d\bar{z} \\
&\quad \leq  \kappa_\epsilon \norm{ v_{\epsilon}\circ r_j^{-1}  }_{H^1(r_j(\Omega_{\epsilon } \cap V_j))^3}^2 \xrightarrow[\epsilon \to 0]{} 0.
\end{split}
\end{equation*}
Hence to prove that  $v_\epsilon \xrightarrow[\epsilon \to 0]{E} v$  we just have to show that 
\begin{equation} \label{Econvergenceofveps}
\lim_{\epsilon \to 0}\norm{v_\epsilon\rvert_\Omega - v}_{L^2(\Omega)^3}=0.
\end{equation}
Recall that   $\{v_\epsilon\rvert_\Omega \}\subset H^1(\Omega)^3$ is  bounded in $H^1$-norm. Select now a sequence $\{v_{\epsilon_n}\}_{n \in \mathbb{N}}$ from the family. By the compact embedding of $H^1(\Omega)^3$ into $L^2(\Omega)^3$ we have that, up to choosing a subsequence, $v_{\epsilon_n}\rvert_\Omega \to v^*$ strongly in $L^2(\Omega)^3$ and $v_{\epsilon_n}\rvert_\Omega \rightharpoonup v^*$ weakly in $H^1(\Omega)^3$ for some $v^* \in H^1(\Omega)^3$. By \Cref{limitisthesolution} we have  that $v^*=S_0 f = v \in \xn(\Omega)$. This shows that for any extracted sequence of the family $\{v_\epsilon\rvert_\Omega - v\}_{\epsilon>0}$, there exist a subsequence such that $\norm{v_{\epsilon_{n_k}}\rvert_\Omega - v}_{L^2(\Omega)^3} \xrightarrow[k \to \infty]{}0$. Thus we can conclude that $\norm{v_\epsilon\rvert_\Omega - v}_{L^2(\Omega)^3} \xrightarrow[\epsilon \to 0]{}0$, which is exactly  \eqref{Econvergenceofveps}.
\end{proof}
\begin{remark}\label{lemma:EE:remark}
The hypothesis of \Cref{lemma:EE:convergence} can be weakened to only require that  $\norm{f_\epsilon}_{L^2(\Omega_\epsilon)^3}$ are uniformly bounded and that the extenstion-by-zero of $f_\epsilon$ (restricted to $\Omega$) weakly converges to $f$ in $L^2(\Omega)^3$ as $\epsilon$ goes to 0, which is a weaker assumption than $f_\epsilon \xrightarrow[\epsilon \to 0]{E}f$.
\end{remark}

Finally we can state and prove the main theorem
\begin{theorem} \label{principale}
Let $\mathcal{A}$ be an atlas in $\mathbb{R}^3$ and $\{\Omega_\epsilon\}_{\epsilon>0}$ be a family of bounded domains of class $C^{1,1}(\mathcal{A})$ converging to a bounded domain $\Omega$ of class $C^{1,1}(\mathcal{A})$ as $\epsilon\to 0$, in the sense that  condition \eqref{assumptions} holds. Suppose that the uniform Gaffney inequality \eqref{unigaff}  holds.
Then $S_\epsilon \xrightarrow[]{C}S_0$  as $\epsilon \to 0$. In particular, spectral stability occurs:  the eigenvalues of the operator $S_{\epsilon}$ converge to the eigenvalues of the operator $S_0$, and   the eigenfunctions of the operator $S_{\epsilon}$ $E$-converge to the eigenfunctions of the operator $S_{0}$ as $\epsilon \to 0$. 
\end{theorem}
\begin{proof}
By \Cref{lemma:EE:convergence} we have that $S_\epsilon \xrightarrow[\epsilon \to 0]{EE}S$. 
Now, suppose that we are given a family of data $\{f_\epsilon\}_{\epsilon >0}$ such that $\norm{f_\epsilon}_{L^2(\Omega_\epsilon)^3} \leq 1$ for all $\epsilon>0$, and extract a sequence $\{f_{\epsilon_n}\}_{n \in \mathbb{N}}$ from it. We have to show that we can always find a subsequence $\epsilon_{n_k}\to 0$ and a function $v \in L^2(\Omega)^3$ such that 
\begin{equation} \label{subseq:E:converges}
S_{\epsilon_{n_k}} f_{\epsilon_{n_k}} \xrightarrow[k \to \infty]{E}v.
\end{equation} 
Possibly passing to a subsequence, we can find a function $f$ to which the restriction to $\Omega$ of the extension-by-zero of $\{f_{\epsilon_n}\}_{n \in \mathbb{N}}$ weakly converge in $L^2(\Omega)^3$. Setting $v:= S_0f$, we can  apply \Cref{lemma:EE:convergence} and Remark~\ref{lemma:EE:remark} to find out that \eqref{subseq:E:converges} holds.

Finally, the spectral stability  is a consequence of the compact convergence of compact operators as stated in Theorem~\ref{vainikko}. 
\end{proof}

\section{Uniform Gaffney Inequalities and applications to fa\-mi\-lies of oscillating boundaries}\label{unifsec}

In this section we give sufficient conditions  that guarantee the validity of a uniform Gaffney inequality of the type \eqref{unigaff} for a family of Lipschitz domains $\Omega_{\epsilon}$, $\epsilon >0$, belonging to the same class $C^{0,1}_M({\mathcal{A}})$.  To do so, we exploit a known relation between Gaffney inequalities and a priori estimates for the Dirichlet Laplacian that we formulate in our setting. We note that one of the two implications (namely, the validity of the Gaffney inequality implies the validity of the a priori estimate) is quite standard. The other one is a bit more involved, hence,  for the sake of completeness, we include a proof.

\begin{theorem}\label{mainequi} Let $\Omega$ be a bounded domain in ${\mathbb{R}}^3$ of class $C^{0,1}_M({\mathcal{A}})$. Then the Gaffney inequality \eqref{gaff}
 holds for all $u \in \xn(\Omega)$ and  a constant $C>0$  independent of $u $ if and only if (the weak, variational) solutions $\varphi \in H^1_0(\Omega)$ to the Poisson problem
\begin{equation}\label{mainequi0}\left\{ 
\begin{array}{ll}
-\Delta \varphi =f,& \ {\rm in}\ \Omega,\\
\varphi =0,&\ {\rm on }\ \partial \Omega ,
\end{array}\right.
\end{equation}
satisfy the a priori estimate 
\begin{equation}\label{apriori}
\| \varphi \|_{H^2(\Omega)}\le \tilde C \| f\|_{L^2(\Omega)}
\end{equation}
for all $f\in L^2(\Omega)$ and a constant $\tilde C>0$ independent of $f$.  In particular, the constants $C$ and $\tilde{C}$ depend on each other, $M$ and ${\mathcal{A}}$. 
\end{theorem}

\begin{proof}
Assume that the a priori estimate \eqref{apriori} holds.  We set $H^1_{\rm \scriptscriptstyle N}(\Omega):=\xn(\Omega)\cap H^1(\Omega)^3$ and  $E(\Omega)=\{ \nabla \varphi:\ \varphi \in H^1_0(\Omega),\ \Delta \varphi \in L^2(\Omega )\} $. By  \cite[Thm.~4.1]{birsol} there exists two linear operators $P:\xn(\Omega)\to H^1_{\rm \scriptscriptstyle N}(\Omega)$ and
$Q  :\xn(\Omega)\to E(\Omega)$ such that $u=Pu+Qu$ for all $u\in \xn(\Omega)$ and such that 
$$
\| Pu \|_{H^1(\Omega)^3}+\| Qu  \|_{L^2(\Omega)^3}+\| \operatorname{div}  Qu\|_{L^2(\Omega)^3}\le C_{BS} \| u\|_{\xn(\Omega)}
$$
for all $u\in \xn(\Omega)$, for some positive constant $C_{BS}$. A careful inspection of the proof of \cite[Thm.~4.1]{birsol} reveals that $C_{BS}$ depends only on 
$M$, ${\mathcal{A}}$. By definition, $Qu=\nabla \varphi $ with $\varphi\in H^1_0(\Omega)$ and $\Delta \varphi\in L^2(\Omega)$. Since we have assumed that \eqref{apriori} holds, then 
$$
\| \varphi\|_{H^2(\Omega)} \le \tilde C\|\Delta\varphi \|_{L^2(\Omega)} =\tilde C  \| \operatorname{div}  Qu  \|_{L^2(\Omega )}\le \tilde C C_{BS}\| u\|_{\xn(\Omega)}
$$
Thus, since $ \| Qu \|_{H^1(\Omega)^3}$ is obviously controlled by $\| \varphi\|_{H^2(\Omega)}$ we deduce that 
$$
\| u    \|_{H^1(\Omega)^3}\le \| Pu\|_{H^1(\Omega)^3}+ \| Qu \|_{H^1(\Omega)^3} \le C\| u\|_{\xn(\Omega)},
$$
for all $u\in \xn(\Omega)$, and \eqref{gaff} is proved. 

Viceversa, assume that \eqref{gaff} holds  and let $\varphi $ be a solution to \eqref{mainequi0}. Since $\nabla \varphi \in \xn(\Omega)$,  by \eqref{gaff}
it follows that 
\begin{equation*}
\begin{split}
\| \nabla \varphi \|_{H^1(\Omega)^3} &\le C( \|\nabla \varphi \|_{L^2(\Omega)^3} + \| \operatorname{curl} \nabla \varphi \|_{L^2(\Omega)^3} +  \|  \operatorname{div} \nabla \varphi \|_{L^2(\Omega)^3} ) \\
& = C ( \| \nabla \varphi \|_{L^2(\Omega)^3} + \|\Delta \varphi \|_{L^2(\Omega)}) \\
& \leq C (c_{\mathcal{P}} \|\Delta \varphi \|_{L^2(\Omega)} +\|\Delta \varphi \|_{L^2(\Omega)}) \\
& \leq C (c_{\mathcal{P}}  +1) \| f\|_{L^2(\Omega)},
\end{split}
\end{equation*}
where we have used \cite[Lemma 1]{ferlam} and  $c_{\mathcal{P}}$ denotes the usual Poincar\'{e} constant.
This, combined with the  Poincar\'{e}'s inequality,  immediately implies \eqref{apriori}.
\end{proof}

\begin{example}\label{counter}
Let $\Omega$ be a bounded domain in $\R^N$  of class $C^1$ such that   around a boundary point (identified here with the origin) is described by the subgraph
$x_N< g(\bar x)$ 
of the $C^1$ function defined by 
$$
 g(x_1, \dots , x_{N-1})= |x_1|/\log |x_1|  
$$
It is proved in \cite[~14.6.1]{Mazya} that for this domain the a priori estimate
\eqref{apriori} does not hold.  Thus, by Theorem~\ref{mainequi} it follows that not even the  Gaffney inequality holds for this domain for $N=3$.
\end{example}

Theorem~\ref{mainequi}  highlights the importance of proving the  a priori estimate \eqref{apriori} and getting information on the constant $\tilde C$.  We do this by following the approach of
Maz'ya and Shaposhnikova~\cite{Mazya} and using the notion of domains $\Omega$ with boundaries $\partial \Omega$ of class $\mathcal{M}^{3/2}_2 (\delta )$. We re-formulate the definition in Maz'ya and Shaposhnikova~\cite[\S~14.3.1]{Mazya}  by using the atlas classes. Here we can treat the general case of domains in  $\R^N$ with $N\geq 2$.

 Note that in this section, following \cite{Mazya} we find it convenient to assume directly that the functions $g_j$ describing the boundary of $\Omega$ as in Definition~\ref{atlas} are extended to the whole of ${\mathbb{R}}^{N-1}$ and belong to the corresponding function spaces defined on ${\mathbb{R}}^{N-1}$.

\begin{defn}\label{multiplier}   Let ${\mathcal{A}}$   be an atlas in ${\mathbb{R}}^N$   and $\delta >0$. We say that a bounded domain  $\Omega$ in  ${\mathbb{R}}^N$  is of class   $\mathcal{M}^{3/2}_2 ( \delta , {\mathcal{A}}   )$ if $\Omega$ is of  class $C^{0,1}({\mathcal{A}}) $ and  the gradients $\nabla g_j$ of the functions $g_j$ describing the boundary of $\Omega$ as in Definition~\ref{atlas}
belong to the space $MW^{1/2  }_2({\mathbb{R}}^{N-1} )$ of Sobolev multipliers with
\begin{equation}\label{multiplier00}
\| \nabla g_j\|_{MW^{1/2}_2({\mathbb{R}}^{N-1} )   }\le \delta \, 
\end{equation} 
for all $j=1,\dots s'$.  We say that  a bounded domain  $\Omega$ in  ${\mathbb{R}}^N$  is of class   $\mathcal{M}^{3/2}_2 ( \delta  )$ if it is of class   $\mathcal{M}^{3/2}_2 (\delta , {\mathcal{A}}  )$  for some atlas ${\mathcal{A}} $.  
\end{defn}

Recall that  $MW^{1/2  }_2({\mathbb{R}}^{N-1} )=\{f\in W^{1/2}_{2, loc}({\mathbb{R}}^{N-1}  ):\  f\varphi \in  W^{1/2}_2({\mathbb{R}}^{N-1}  )\ {\rm for\ all}\  \varphi \in W^{1/2}_2({\mathbb{R}}^{N-1}  )\} $  and that 
$\|  f\|_{MW^{1/2}_2({\mathbb{R}}^{N-1} )   }=  \sup \{ \|   f\varphi    \|_{ W^{1/2}_2({\mathbb{R}}^{N-1}  )  }:\|   \varphi \|_{ W^{1/2}_2({\mathbb{R}}^{N-1}  )  }=1\}$, where $W^{1/2}_2({\mathbb{R}}^{N-1}  )$ denotes the standard Sobolev space with fractional order of smoothness $1/2$ and index of summability $2$ (for simplicity, in \eqref{multiplier00} we use the the same symbol for the norm of a vector field).

\begin{remark}
Note that by \cite[Thm.~4.1.1]{Mazya} there exists $c>0$ depending only on $N$ such that the functions $g_j$ in Definition~\ref{multiplier} satisfy the estimate $\| \nabla g_j \|_{L^\infty(\mathbb{R}^{N-1})} \le c \| \nabla g_j\|_{MW^{1/2}_2({\mathbb{R}}^{N-1} )   }\le c\delta $, see also \cite[Thm.~14.6.4]{Mazya}. Thus, if $\Omega$ is of class   $\mathcal{M}^{3/2}_2 (\delta ,  {\mathcal{A}}   )$  then it is also of class  $C^{0,1}_M({\mathcal{A}}) $ with $M=c\delta$. 
\end{remark}

The following theorem is a reformulation of \cite[Thm.~14.5.1]{Mazya}

\begin{theorem}\label{Mazyaapriori}   Let ${\mathcal{A}}$   be an atlas in ${\mathbb{R}}^N$.    If  $\Omega $ is a bounded domain  of class  $\mathcal{M}^{3/2}_2 (\delta , {\mathcal{A}}   )$  for some $\delta $  sufficiently small  (depending only on $N$) then  the a priori estimate \eqref{apriori} holds for some constant $\tilde C$ depending only on $N$ and ${\mathcal{A}}$.  
\end{theorem}

By  \cite[Corollaries.~14.6.1,~14.6.2]{Mazya} it is possible to prove the following theorem based on the condition (\ref{dini}) from \cite[(14.6.9)]{Mazya} . Here, given an atlas  $\mathcal{A}=(\rho,s,s',\{V_j\}_{j=1}^s, \{r_j\}^s_{j=1} )$, by a refinement of $\mathcal{A}$ we mean an atlas
of the type  $\widetilde{\mathcal{A}}=(\tilde \rho,\tilde s,\tilde s',\{\tilde V_j\}_{j=1}^{\tilde s}, \{\tilde r_j\}^{\tilde s}_{j=1} )$ where $\tilde \rho \le \rho$, $s\le \tilde s$, $s'\le \tilde s'$, $\cup_{j=1}^{\tilde s}{\tilde V_j}= \cup_{j=1}^{\ s}V_j$, $\{\tilde r_j\}^{\tilde s}_{j=1}\subset \{ r_j\}^{ s}_{j=1}$, which can be thought as an atlas constructed from $\mathcal{A}$ by replacing each cuboid $V_{j}=r_j(  W_{j}\times ]a_{N,j}, b_{N,j}[   )  $ by a finite number of cuboids of the form  $\widetilde V_{j,l}=r_j(  \widetilde W_{j,l}\times ]a_{N,j}, b_{N,j}[   )  $, $l=1,\dots m_j$, where $W_j=\cup_{l=1}^{m_j} \widetilde W_{j,l}$.

\begin{theorem} \label{localmazya}  Let ${\mathcal{A}}$   be an atlas in ${\mathbb{R}}^N$ and let $\Omega $ be a bounded domain of class   $C^{0,1}_M({\mathcal{A}}) $.
Let $\omega $ be a (non-decreasing)  modulus of continuity for the gradients $\nabla g_{j}$  of the functions $g_j$ describing the boundary of $\Omega$,  that is 
\begin{equation}\label{modulus}
|\nabla g_j(\bar x)-\nabla g_j(\bar y)|\le \omega (|\bar x-\bar y |)
\end{equation}
for all $\bar x,\bar y\in {\mathbb{R}}^{N-1}$. 
Assume that there exists $D>0$ such that the  function $\omega$ satisfies the following condition\footnote{Here only the integrability at zero really matters and one could consider integrals defined in a neigborhhood of zero.}
\begin{equation}\label{dini}
\int_{0}^{\infty }\left(\frac{\omega (t)}{t}\right)^2dt \le D\, .
\end{equation}
Then there exists $C>0$ depending only on  $N$, ${\mathcal{A}}$, $D$ such that if $M\le C\delta $ then, possibly replacing  the atlas ${\mathcal{A}}$ with  
a  refinement of ${\mathcal{A}}$,  $\Omega$ is of class $\mathcal{M}^{3/2}_2 (\delta, {\mathcal{A}} )$. 
\end{theorem}

\begin{proof}
We begin with the case $N\geq 3$.   By \cite[Cor.~14.6.1]{Mazya} there exists $c>0$ depending only on $N$ such that if $x=(\bar x, g_j(\bar x))\in \partial \Omega$ is any point of the boundary represented in local charts by a profile function $g_j$  and 
 the following inequality 
\begin{equation}\label{localmazya0}
\lim_{\rho \to 0}\biggl(\sup_{E\subset B_{\rho}(\bar x) } \frac{ \| D_{3/2}(g_j , B_{\rho} ) \|_{L^2(E)}  }{   |E|^{ \frac{N-2}{2(N-1)}   } }  + \| \nabla g_j\|_{L^{\infty }(B_{\rho}(\bar x) ) }\biggr)\le c\delta 
\end{equation}
is satisfied,  then, possibly replacing the atlas ${\mathcal{A}}$ with  
a  refinement of its, $\Omega$ is of class $\mathcal{M}^{3/2}_2 (\delta , {\mathcal{A}}  )$. 
Here $|E|$ denotes the $N-1$-dimensional Lebesgue measure of $E$, 
$$
D_{3/2}(g_j , B_{\rho} )(\bar x) =\biggl(    \int_{B_{\rho} (\bar x)}    | \nabla g_j(\bar x) -\nabla g_j(\bar y )  |^2|\bar x-\bar y|^{-N}d\bar y  \biggr)^{1/2},
$$
and $B_{\rho}(\bar x)$ the  ball in ${\mathbb{R}}^{N-1}$ of radius $\rho$ and centre $\bar x$. 
 We refer to \cite[\S 14.7.2]{Mazya} for the local characterization
of the boundaries of domains of class  $\mathcal{M}^{3/2}_2 (\delta , {\mathcal{A}}  )$. 

We have
\begin{eqnarray}
\int_E\int_{B_{\rho}(\bar x)}  | \nabla g_j(\bar x) -\nabla g_j(\bar y )  |^2|\bar x-\bar y|^{-N}d\bar y d\bar x \le \int_E\int_{B_{\rho}(\bar x) } \frac{\omega ^2(|\bar x-\bar y|)}{|\bar x-\bar y|^N}d\bar y d\bar x\nonumber  \\
= \int_E\int_{B_{\rho}(0) } \frac{\omega^2 (|\bar h|)}{|\bar h|^N} d\bar hd\bar x 
= \sigma_{N-2} |E| \int_0^{\rho }\left|\frac{\omega(t)}{t}\right|^2dt\le \sigma_{N-2} D|E| \, .
%\le  \sigma_{N-2}\sigma_{N-1} \rho^{N-1}D\rho^{N-1}/ N
\end{eqnarray}
Here  $\sigma_m$ denotes the $m$-dimensional measure of the $m$-dimensional unit sphere. Thus
$$
 \frac{ \| D_{3/2}(g_j , B_{\rho} ) \|_{L^2(E)}  }{   |E|^{ \frac{N-2}{2(N-1)}   } } \le (\sigma_{N-2} D)^{1/2}|E|^{\frac{1}{2(N-1)}}=O(\rho ^{1/2}),
$$
hence
\begin{equation}\label{localmazya1}
\frac{ \| D_{3/2}(g_j , B_{\rho} ) \|_{L^2(E)}  }{   |E|^{ \frac{N-2}{2(N-1)}   } } \le c\delta , 
\end{equation}
provided $\rho$ is sufficiently small.  Thus, inequality \eqref{localmazya0} follows if we assume directly that 
$\| \nabla g_j\|_{L^{\infty }(B_{\rho} ) }  \le c\delta$.

In the case $N=2$,  by \cite[Cor.~14.6.1]{Mazya} it suffices to replace  \eqref{localmazya0} by the following inequality
\begin{equation}\label{localmazyalog}
\lim_{\rho \to 0}\biggl(\sup_{E\subset B_{\rho}(\bar x) }  \| D_{3/2}(g_j , B_{\rho} ) \|_{L^2(E)}   |\log |E||^{1/2}   + \| \nabla g_j\|_{L^{\infty }(B_{\rho} (\bar x)  ) }\biggr)\le c\delta 
\end{equation}
and use the same argument as above.
\end{proof}

By combining Theorems~\ref{Mazyaapriori} and \ref{localmazya}, we deduce the validity of the following  result

\begin{cor} \label{aprioricorol} Under the same assumptions of Theorem~\ref{localmazya},  there exists $\tilde C>0$ depending only on  $N$, ${\mathcal{A}}$, $D$ such that if $M< \tilde C^{-1} $ then the a priori estimate \eqref{apriori} holds.
\end{cor}

Finally, by Theorem~\ref{mainequi} and Corollary~\ref{aprioricorol} we deduce the following result ensuring the validity of uniform Gaffney inequality that can be used in our spectral stability results.

\begin{cor} \label{aprioricorol2} Under the same assumptions of Theorem~\ref{localmazya} with $N=3$,  there exists $ C>0$ depending only on ${\mathcal{A}}$ and $D$ such that if $M<  C^{-1} $ then  the Gaffney inequality \eqref{gaff}. 
\end{cor}

\subsection{Applications to families of oscillating boundaries}

It is clear that in order to apply Theorem~\ref{localmazya} and Corollaries~\ref{aprioricorol}, \ref{aprioricorol2}, it suffices to assume that the gradients $\nabla g_j$ of the  functions $g_j$ describing the boundary of a domain $\Omega$ as in Definition~\ref{atlas} are of class $C^{0,\beta}$ with $\beta \in ]1/2, 1]$, that is 
\begin{equation} \label{holdernablag}
|\nabla g_j(\bar x)-\nabla g_j(\bar y)|\le  K  |\bar x-\bar y |^{\beta}
\end{equation}
for some positive constant  $K$ and  all $\bar x, \bar y\in W_j$, and that the functions $g_j$ have sufficiently  small Lipschitz constants.  As we have already mentioned, in principle, the second condition is not  a big 
obstruction to the application of these results, since for a domain of class $C^1$ one can find a sufficiently refined atlas, adapted to the tangent planes of a finite number of boundary points, such that the $C^1$ norms, hence the Lipschitz constants, of the profile functions $g_j$ are arbitrarily close to zero.
Thus, we can apply our results to uniform classes of domains of class $C^{1,\beta}$ since condition \eqref{dini} would be satisfied exactly because $\beta >1/2$ (as we have said, here what matters is the behaviour of the modulus of continuity $\omega (t)$ for $t$ close to zero  and one can assume directly that $\omega(t)$ is constant for $t$ big enough). 

Thus, we can prove the following result. 
Note that here the domains $\Omega_\epsilon$ are assumed to be of class $C^{1,1}$ and that they belong to the uniform class $C^{1,\beta}_K(\mathcal{A})$ with $K>0$ fixed, which in particular implies the validity of \eqref{holdernablag} for  all functions $g_{\epsilon,j}$ and all $\epsilon>0$. (Recall that the operators $S_{\epsilon}$ are defined in the beginning of Section~\ref{sezione4}.)

\begin{theorem} \label{principaleholder}
Let $\mathcal{A}$ be an atlas in $\mathbb{R}^3$ and $\{\Omega_\epsilon\}_{\epsilon >0}$ be a family of bounded domains of class $C^{1,1}(\mathcal{A})$ converging to a bounded domain $\Omega$ of class $C^{1,1}(\mathcal{A})$ as $\epsilon\to 0$, in the sense that  condition \eqref{assumptions} holds. Suppose that  $ \Omega $ is of class $C^{0,1}_M(\mathcal{A})$ with $M$ small enough as in Corollary~\ref{aprioricorol2}. Suppose also that all domains $\Omega_{\epsilon}$ are of class $C^{1,\beta}_K(\mathcal{A})$ with the same parameters $\beta \in ]1/2,1]$ and $K>0$. Then the uniform Gaffney inequality \eqref{unigaff}  holds  provided $\epsilon$ is small enough. Moreover,
$S_\epsilon \xrightarrow[]{C}S$  as $\epsilon \to 0$. In particular, spectral stability occurs:  the eigenvalues of the operator $S_{\epsilon}$ converge to the eigenvalues of the operator $S_{0}$, and   the eigenfunctions of the operator $S_{\epsilon}$ $E$-converge to the eigenfunctions of the operator $S_{0}$ as $\epsilon \to 0$. 
\end{theorem}

\begin{proof}
Since $\Omega_\epsilon$ converges to  $\Omega$ as $\epsilon \to 0$  in the sense that  condition \eqref{assumptions} holds, it follows that the gradients of the functions $g_{\epsilon , j}$ describing the boundary of $\Omega_{\epsilon }$ converge uniformly to the gradients  of the functions $g_{ j}$ describing the boundary of $\Omega$. Thus,   $ \Omega_{\epsilon} $ is of class $C^{0,1}_M(\mathcal{A})$ provided $\epsilon $ is small enough. By the discussion  above, Corollary~\ref{aprioricorol2} is applicable and  the uniform Gaffney inequality \eqref{unigaff}  holds provided $\epsilon$ is small enough.
Then the last part of the statement follows by Theorem \ref{principale}.
\end{proof} 

A prototype for the classes of domains under discussion is given by domains designed by  profile functions
often used in homogenization theory, in particular in the study of thin domains. Namely, assume  that one of the profile functions $g_{\epsilon,j}$, call it $g_\epsilon$, is 
of the form 
\begin{equation}
g_\epsilon(\bar x )=\epsilon^{\alpha}b(\bar x/\epsilon )
\end{equation}
for some function $b$ of class $C^{1,1}(\mathbb{R}^2 )$ and $\alpha >0$, and  assume that the gradient of $b$ is bounded. If $\omega_{\nabla b}$ is a (non-decreasing) modulus of continuity  of  $\nabla b$, then we have 
$$
|\nabla g_\epsilon(\bar x )-\nabla g_\epsilon(\bar y)| = \epsilon^{\alpha-1}|  \nabla b(\bar x/\epsilon )- \nabla b(\bar y/\epsilon )|\le \epsilon^{\alpha-1}\omega_{\nabla b}
\biggl(\frac{\bar x-\bar y }{\epsilon} \biggr)\, ,
$$ 
hence the function $\omega $ to be considered in \eqref{modulus} is given by $\omega(t)= \epsilon^{\alpha-1}\omega_{\nabla b}(t/\epsilon )$. Observe that 
\begin{equation}\label{lastint}
\int_0^{\infty}   \biggl(\frac{\omega(t)}{t}\biggr)^2dt=\epsilon^{2\alpha -2}\int_0^{\infty } \left( \frac{\omega_{\nabla b}(t/\epsilon ) }{t} \right)^2 dt=
\epsilon^{2\alpha -3}\int_0^{\infty } \left( \frac{\omega_{\nabla b}(s ) }{s} \right)^2 ds\, .
\end{equation}
Moreover, since $b$ is assumed to be of  class $C^{1,1}$, we have that  $\omega_{\nabla b}(t)\le ct $ for $t$ in a neighborhhood of zero. Thus, if $\alpha \geq 3/2$ and $\epsilon_0$ is any fixed positive constant, it follows that 
that 
\begin{equation}\label{lastint00}
\sup_{\epsilon\in  ]0, \epsilon_0]}
\epsilon^{2\alpha -3}\int_0^{\infty } \left( \frac{\omega_{\nabla b}(s ) }{s} \right)^2 ds \ne \infty \, .
\end{equation}
Since the gradient of $g_\epsilon$ is arbitrarily  close to zero for $\epsilon$ sufficiently small, we have that Theorem~\ref{localmazya} and Corollaries~\ref{aprioricorol}, \ref{aprioricorol2} are applicable and the Gaffney inequality \eqref{unigaff} holds for all $\epsilon$ sufficiently small, with a constant $C>0$ independent of $\epsilon$.   The same arguments can be applied to families of profile functions of the type 
$$
g_\epsilon(\bar x )=\epsilon^{\alpha}b(\bar x/\epsilon )\psi(\bar x)
$$
where $b$ is as above and $\psi $ is a fixed $C^{1,1}$ function with bounded gradient. Thus, we can state the following stability result concerning a local perturbation of a domain $\Omega$.

\begin{theorem}\label{mainhomo}  Let $W$ be a bounded open rectangle in  ${\mathbb{R}}^2$, $b\in C^{1,1}(\mathbb{R}^{2})$ with bounded gradient, $b\ge 0$,  and $\psi \in C^{1,1}_c(W)$, $\alpha > 3/2$. Assume that $\Omega $  and $\Omega_{\epsilon}$, $\epsilon > 0$ are domains of class $C^{1,1}$ in ${\mathbb{R}}^3$ satisfying the following condition:
\begin{itemize}
\item[(i)]  $\Omega \cap (W\times ]-1,1[)=\{(\bar x, x_3)\in {\mathbb{R}}^3:\ \bar x\in W,\ -1<x_3<0 \} $;
\item[(ii)]  $\Omega_{\epsilon} \cap (W\times ]-1,1[ )=\{(\bar x, x_3)\in {\mathbb{R}}^3:\ \bar x\in W,\ -1<x_3<\epsilon^{\alpha}b(\bar x/\epsilon )\psi(\bar x) \} $
where $b\in C^{1,1}(\mathbb{R}^{2})$ has bounded gradient, and $\psi \in C^{1,1}_c(W)$;
\item[(iii)] $\Omega \setminus  (W \times ]-1,1[) =  \Omega_{\epsilon } \setminus  (W \times ]-1,1[) $;
\end{itemize}
Then  the family  $\{\Omega_{\epsilon} \}_{\epsilon >0}$ converges to $\Omega $  in the sense that  condition \eqref{assumptions} holds. Moreover,  the uniform Gaffney inequality \eqref{unigaff} holds and $S_\epsilon \xrightarrow[]{C}S$  as $\epsilon \to 0$.   In particular, spectral stability occurs:  the eigenvalues of the operator $S_{\epsilon}$ converge to the eigenvalues of the operator $S_0$, and   the eigenfunctions of the operator $S_{\epsilon}$ $E$-converge to the eigenfunctions of the operator $S_{0}$ as $\epsilon \to 0$. 
\end{theorem}

\begin{proof}
By assumptions, the domains $\Omega$ and $\Omega_{\epsilon}$ belong to the same atlas class  $C^{1,1}({\mathcal{A}})$ for a   suitable atlas ${\mathcal{A}}$, and $W\times ]-1,1[$ is one of the local charts of ${\mathcal{A}}$. In particular, the profile functions describing the boundaries of $\Omega$ and $\Omega_{\epsilon}$ in that chart are given by  $g(\bar x)= 0$ and  $g_{\epsilon}= \epsilon^{\alpha}b(\bar x/\epsilon )\psi(\bar x)$ for all $\bar x\in W$.   

  As in the proof of \cite[Thm.~7.4]{arrlam}, if $\tilde \alpha \in ]3/2, \alpha[$ is fixed  then one can easily check that  conditions \eqref{assumptions} are satisfied with
  $k_{\epsilon}=\epsilon^{2\tilde \alpha /3}$.   By \eqref{apriori} and the discussion above, it follows that the Gaffney inequality \eqref{gaff} holds with a constant $C$ independent of $\epsilon$, provided $\epsilon$ is sufficiently small. To complete the proof it suffices to apply  Theorem~\ref{principale}. 
\end{proof}

\begin{remark}\label{finalremark} It is clear that condition \eqref{lastint00} is satisfied also in the case $\alpha =3/2$. Thus the uniform Gaffney inequality  \eqref{unigaff} holds also in the case  $\alpha =3/2$ in Theorem~\ref{mainhomo}. However, in this case the convergence of $\Omega_{\epsilon}$ to $\Omega$ in the sense of \eqref{assumptions}  is not guaranteed hence we cannot directly deduce that we have spectral stability. Thus, another method has to be used in the analysis of the stability problem for $\alpha =3/2$. For example, in the case of non-constant periodic functions $b$ one could use the unfolding method as in \cite{casado}, adopted also in  \cite{arrferlam, arrlam, ferralamb, ferlam}: in those papers, for $\alpha =3/2$ we have  spectral instability in the sense that the limiting problem differs from the given problem in $\Omega$ by a strange term appearing in the boundary conditions (as often happens in homogenization problems). We plan to discuss the details of this problem for the $curl curl $ operator in a forthcoming paper, but we can already mention that a preliminary formal analysis would indicate  that no strange limit appears in the limiting problem for $\alpha =3/2$. 
 On the other hand, at the moment we are not able to formulate any conjecture for the case $\alpha <3/2$ although, on the base of the  results of \cite{casado} concerning the  Navier-Stokes system, a 
degeneration phenomenon (to
Dirichlet boundary conditions) 
could not be excluded. 
\end{remark}

{\bf Acknowledgments:} The authors are also very thankful to Dr. Francesco Ferraresso and Prof. Ioannis G. Stratis for useful discussions and references.
The authors acknowledge financial support from the research project BIRD191739/19  ``Sensitivity analysis of partial differential equations in the mathematical theory of electromagnetism'' of the University of Padova. 
The authors are members of the Gruppo Nazionale per l'Analisi Matematica, la Probabilit\`a e le loro Applicazioni (GNAMPA) of the Istituto Nazionale di Alta Matematica (INdAM).

\vspace{24pt}

\noindent Pier Domenico Lamberti\\
Dipartimento di  tecnica e gestione dei sistemi industriali (DTG)\\
University of Padova\\
Stradella S. Nicola 3 \\
36100 Vicenza\\ 
Italy\\
 e-mail: lamberti@math.unipd.it\\

\noindent  Michele Zaccaron\\
Dipartimento di Matematica `Tullio Levi-Civita'\\
University of Padova\\
 Via Trieste 63\\
35121 Padova\\ 
Italy\\
 e-mail:  zaccaron@math.unipd.it\\


\begin{thebibliography}{100} 


\bibitem{arcalo} Arrieta J.,  Carvalho A. N.,  Losada-Cruz G., Dynamics in dumbell domains I. Continuity of the set of equilibria, J. Diff. Eq.231, 551-597, (2006).


\bibitem{arrferlam}  Arrieta J.,  Ferraresso F.,  Lamberti P.D.,  Boundary homogenization for a triharmonic intermediate problem. Math. Methods Appl. Sci. 41 (2018), no. 3, 979-985.

\bibitem{arrlam} Arrieta, J., Lamberti, P.D., Higher order elliptic operators on variable domains. Stability results and boundary oscillations for intermediate problems. J. Differential Equations, Vol.263 No.7, 2017.


\bibitem{arrvil}  Arrieta, J.,  Villanueva-Pesqueira, M., Elliptic and parabolic problems in thin domains with doubly weak oscillatory boundary. Commun. Pure Appl. Anal. 19 (2020), no. 4, 1891-1914.

\bibitem{birsol} Birman, M.Sh., Solomyak, M.Z., The Maxwell operator in domains with a nonsmooth boundary. Sibirsk. Mat. Zh. 28 (1987), no. 1, pag. 23-36.
Siberian Math. J., 28:1 (1987), 12-24.


\bibitem{bre} Brezis, H., Functional Analysis, Sobolev Spaces and Partial Differential Equations. Springer, 2010.

\bibitem{bur} Burenkov, V.I., Sobolev spaces on domains. Teubner-Texte zur Mathematik, Vol.137, 1998

\bibitem{burlam} Burenkov, V.I., Lamberti, P.D., Sharp spectral stability estimates via the Lebesgue measure of domains forhigher order elliptic operators. Rev. Mat. Complut.25, 435–457, 2012

\bibitem{capevi} Cardone, G., Perugia, C., Villanueva Pesqueira, M.,  Asymptotic behavior of a Bingham flow in thin domains with rough boundary. Integral Equations Operator Theory 93 (2021), no. 3, Paper No. 24, 26 pp.

\bibitem{capi} Carvalho A., Piskarev S., A general approximation scheme for attractors of abstract parabolic problems. Numer. Funct. Anal.Optim.27, no. 7-8, 785-829, (2006).

\bibitem{casado} Casado-D\'{\i}az J., Luna-Laynez M.,  Su\'{a}rez-Grau F.J., Asymptotic behavior of a viscous fluid with slip boundary
conditions on a slightly rough wall, Math. Models Methods Appl. Sci. 20 (1) (2010) 121-156.

\bibitem{ces} Cessenat, M.,  Mathematical Methods in Electromagnetics, Linear Theory and Applications, Series on Advances in Mathematics for Applied Sciences - Vol. 41, World Scientific Publishing, Singapore, 1996.


\bibitem{co91} Costabel, M.,  A remark on the regularity of solutions of Maxwell's equations on Lipschitz domains.
Math. Methods Appl. Sci. 12 (1990), no. 4, 365-368. 

\bibitem{cocoercive} Costabel, M., A coercive bilinear form for Maxwell's equations. J. Math. Anal. Appl. 157 (1991), no. 2, 527-541. 

\bibitem{coda} Costabel, M.,  Dauge M., Maxwell and Lam\'{e} eigenvalues on polyhedra, Math. Methods Appl. Sci. 22 (1999),  243-258. 

\bibitem {coda2} Costabel, M.,  Dauge, M., Maxwell eigenmodes in product domains. in  Maxwell's Equations: Analysis and Numerics, edited by Ulrich Langer, Dirk Pauly and Sergey Repin, Berlin, Boston: De Gruyter, 2019, pp. 171-198. 

\bibitem{dalamu} Dalla Riva M., Lanza de Cristoforis M., Musolino P.,  Singularly perturbed boundary value problems. A functional analytic approach. Springer. xvi, 672 p. (2021).

\bibitem{dali3} Dautray, R., Lions, J.-L.,  Mathematical Analysis and Numerical Methods for Science and Technology: Vol. 3,  Spectral Theory and Applications, Springer-Verlag, Berlin, 1990.

\bibitem{ferra} Ferraresso, F., On the spectral instability for weak intermediate triharmonic problems. To appear in Mathematical Methods in the Applied Sciences.

\bibitem{ferralamb} Ferraresso, F., Lamberti, P.D., On a Babu\v{s}ka paradox for polyharmonic operators: spectral stability and boundary homogenization for intermediate problems. Integral Equations Operator Theory 91 (2019), no. 6, Paper No. 55, 42 pp. 


\bibitem{ferlam} Ferrero, A., Lamberti, P.D., Spectral stability for a class of fourth order Steklov problems under domain perturbations. Calc. Var. Vol.58 No.33, 2019

\bibitem{ferlam2} Ferrero, A., Lamberti, P.D., Spectral stability of the Steklov problem.  arXiv:2103.04991.

\bibitem{fil} Filonov, N., Principal singularities of the magnetic field component in resonators with a boundary of a given class of smoothness. Algebra i Analiz, Vol.9 No.2 (1997), pp.241-255.




\bibitem{gira} Girault, V., Raviart P.-A.,  Finite Element Approximation of the Navier-Stokes Equations, Lecture Notes in Mathematics - No. 749, Springer, Berlin, 1981.

\bibitem{hanyak}  Hanson, G.W., Yakovlev, A.B.,  Operator Theory for Electromagnetics, Springer-Verlag New York, 2002. 

\bibitem{he1} Hagemann, F.,  Hettlich, F. Application of the second domain derivative in inverse electromagnetic scattering. Inverse Problems 36 (2020), no. 12, 125002, 34 pp.
 
 
\bibitem{he2} Hagemann, F., Arens, T.,  Betcke, T.,  Hettlich, F.  Solving inverse electromagnetic scattering problems via domain derivatives. Inverse Problems 35 (2019), no. 8, 084005, 20 pp. 

\bibitem{henry}  Henry, D. Perturbation of the boundary in boundary-value problems of partial differential equations. With editorial assistance from Jack Hale and Ant\^{o}nio Luiz Pereira. London Mathematical Society Lecture Note Series, 318. Cambridge University Press, Cambridge, 2005.


\bibitem{he3} Hettlich, F.  The domain derivative of time-harmonic electromagnetic waves at interfaces. 
Math. Methods Appl. Sci. 35 (2012), no. 14, 1681-1689. 

\bibitem{hira} Hirakawa, K., Denki Rikigaku. Baifukan, Tokyo (1973) (in Japanese).


\bibitem{jimbo} Jimbo, S., Hadamard variation for electromagnetic frequencies. Geometric properties for parabolic and elliptic PDE's, 179-199, Springer INdAM Ser., 2, Springer, Milan, 2013. 



\bibitem{kihe} Kirsch, A.,  Hettlich, F., 
The mathematical theory of time-harmonic Maxwell's equations.
Expansion-, integral-, and variational methods. Applied Mathematical Sciences, 190. Springer, Cham, 2015.


\bibitem{kswy} Kristensson, G., Stratis, I. G., Wellander, N., Yannacopoulos, A. N., The exterior Calder\'{o}n operator for non-spherical objects, SN Partial Differ. Equ. Appl., 1 (2020), 6.

\bibitem{lamstra} Lamberti P.D., Stratis, I.G.,   On an interior Calder\'{o}n  operator and  a related Steklov eigenproblem for Maxwell's equations,   SIAM J. Math. Anal. 52 (2020), no. 5, 4140-4160. 

\bibitem{lamzac}  Lamberti, P.D.,  Zaccaron, M.,  Shape sensitivity analysis for electromagnetic cavities. Mathematical Methods in the Applied Sciences
Volume 44, Issue 13 p. 10477-10500.


\bibitem{Mazya}  Maz'ya, V.G., Shaposhnikova, T.O.,  Theory of Sobolev multipliers. With applications to differential and integral operators. Grundlehren der Mathematischen Wissenschaften [Fundamental Principles of Mathematical Sciences], 337. Springer-Verlag, Berlin, 2009. xiv+609 

\bibitem{monk} Monk, P.,  Finite Element Methods for Maxwell's Equations, Clarendon Press, Oxford, 2003.

\bibitem{ned} N\'edel\'ec, J.-C.,  Acoustic and Electromagnetic Equations, Integral Representations for Harmonic Problems, Applied Mathematical Sciences - Vol. 144, Springer-Verlag, New York, 2001.


\bibitem{pauly}  Pauly, D., On the Maxwell constants in 3D. Math. Methods Appl. Sci. 40 (2017), no. 2, 435-447. 

\bibitem{PrFil15} Prokhorov, A.; Filonov, N. Regularity of electromagnetic fields in convex domains. J. Math. Sci. (N.Y.) 210 (2015), no. 6, 793-813. 

\bibitem{rsy} Roach, G. F., Stratis, I. G., Yannacopoulos, A. N., Mathematical Analysis of Deterministic and Stochastic Problems in Complex Media Electromagnetics, Princeton Series in Applied Mathematics, Princeton University Press, Princeton, NJ, 2012.


\bibitem{Vai} Vainikko, G. M., Regular convergence of operators and the approximate solution of equations, Math. Anal. 16, 5-53 (1979).


\bibitem{weber} Weber, C. Regularity theorems for Maxwell's equations. Math. Methods Appl. Sci. 3 (1981), no. 4, 523-536. 

\bibitem{web} Weber, Ch., A local compactness theorem for Maxwell's equations. Math. Methods Appl. Sci. 2 (1980), 12-25.

\bibitem{yin} Yin, H.M., An eigenvalue problem for curlcurl operators.  Can. Appl. Math. Q. 20 (2012), no. 3, 421-434. 


\bibitem{zhang} Zhang, Z., Comparison results for eigenvalues of curl\!\! curl operator and Stokes operator, Z. Angew. Math. Phys. 69-104 (2018).





%\bibitem{abste} Abramowitz, M., Stegun, I. A., {\em Handbook of mathematical functions with formulas, graphs, and mathematical tables}, National Bureau of Standards Applied Mathematics Series Vol. 55, U.S. Government Printing Office, Washington, D.C., 1964.

%\bibitem{alberti}  Alberti, G.S., Capdeboscq, Y., Elliptic regularity theory applied to time harmonic anisotropic Maxwell's equations with less than Lipschitz complex coefficients. SIAM J. Math. Anal. 46 (2014), no. 1, 998-1016. 

%\bibitem{ammari} Ammari, H., Bao G., Wood, A.W.,  An integral equation method for the electromagnetic scattering from cavities.  Math. Methods Appl. Sci. 23 (2000), no. 12, 1057-1072.


%\bibitem{bauer} Bauer, S., Pauly, D., Schomburg, M.,   The Maxwell compactness property in bounded weak Lipschitz domains with mixed boundary conditions. SIAM J. Math. Anal. 48 (2016), no. 4, 2912-2943.


%\bibitem{bubu} Bucur, D., Buttazzo, G. Variational methods in shape optimization problems. Progress in Nonlinear Differential Equations and their Applications, 65. Birkh\"{a}user Boston, Inc., Boston, MA, 2005.




%\bibitem{acl} Assous, F., Ciarlet, P., Labrunie, S., {\em Mathematical Foundations of Computational Electromagnetism}, Applied Mathematical Sciences - Vol. 198, Springer International Publishing AG, part of Springer Nature, Cham, Switzerland, 2018.

%\bibitem{Auch2006} Auchmuty, G. Spectral characterization of the trace spaces $H^s(\partial \Omega )$, SIAM J. Math. Anal. 38 (2006), 894-905.

%\bibitem{bogosel}  Bogosel, B. The Steklov spectrum on moving domains, Appl. Math. Optim. 75 (2017), 1-25.

%\bibitem{boumaam} Boulmezaoud, T.-Z., Maday, Y., Amari, T., On the linear force-free fields in bounded and unbounded three-dimensional domains, ESAIM Math. Model. Numer. Anal. 33 (1999), 359-393.

%\bibitem{bucosh} Buffa, A., Costabel, M., Sheen, D., On traces for $H(\cu, \Omega)$ in Lipschitz domains, J. Math. Anal. Appl. 276 (2002), 845-867.


%\bibitem{Buoso2015} Buoso, D., Shape differentiability of the eigenvalues of elliptic systems. Integral methods in science and engineering, 91-97, Birkh\"{a}user/Springer, Cham, 2015.

%\bibitem{bula2015} Buoso, D., Lamberti, P.D., On a classical spectral optimization problem in linear elasticity. New trends in shape optimization, 43-55, Internat. Ser. Numer. Math., 166, Birkh\"{a}user/Springer, Cham, 2015.

%\bibitem{bulareis} Buoso, D., Lamberti, P.D., Shape sensitivity analysis of the eigenvalues of the Reissner-Mindlin system. SIAM J. Math. Anal. 47 (2015), no.1, 407-426.


%\bibitem{buopro} Buoso, D., Provenzano, L.,  A few shape optimization results for a biharmonic Steklov problem, J. Differential Equations 259 (2015), 1778-1818.

%\bibitem{cacomemo} Cakoni, F., Colton, D., Meng, S., Monk, P., Stekloff eigenvalues in inverse scattering, SIAM J. Appl. Math. 76 (2016), 1737-1763.

%\bibitem{calamo} Caman\~o, J., Lackner, C., Monk, P., Electromagnetic Stekloff eigenvalues in inverse scattering, SIAM J. Math. Anal. 49 (2017), 4376-4401.



%\bibitem{cocomo} Cogar, S.,  Colton, D., Monk, P., Eigenvalue problems in inverse electromagnetic scattering theory, in Maxwell's Equations, Analysis and Numerics,  Radon Series on Computational and Applied Mathematics Vol. 24,  U. Langer, D. Pauly and S.I Repin (Eds.), De Gruyter,  Berlin/Boston, 2019. 



%\bibitem{dalla}  Dalla Riva, M., Provenzano, L. On vibrating thin membranes with mass concentrated near the boundary: an asymptotic analysis, SIAM J. Math. Anal. 50 (2018), 2928-2967.





%\bibitem{ferlam} Ferrero, A., Lamberti, P. D., Spectral stability for a class of fourth order Steklov problems under domain perturbations, Calc. Var. Partial Differential Equations (2019) 58:33 (57 pages).




%\bibitem{henrot} Henrot, A. Extremum problems for eigenvalues of elliptic operators. Frontiers in Mathematics. Birkh\"{a}user Verlag, Basel, 2006





%\bibitem{girpol} Girouard, A., Polterovich, I., Spectral geometry of the Steklov problem, J. Spectr. Theory 7 (2017), 321-359.

%\bibitem{hey} Heyden S., Ortiz M., Functional optimality of the sulcus pattern of the human brain, Mathematical Medicine and Biology: A Journal of the IMA (2018) doi: 10.1093/imammb/dqy007 (15 pages).

%\bibitem{kihe} Kirsch, A., Hettlich, F., {\em The Mathematical Theory of Time-Harmonic Maxwell's Equations, Expansion-, Integral-, and Variational Methods}, Applied Mathematical Sciences - Vol. 190, Springer International Publishing, Cham, Switzerland, 2015.

%\bibitem{kswy} Kristensson, G., Stratis, I. G., Wellander, N., Yannacopoulos, A. N., The exterior Calder\'{o}n operator for non-spherical objects, (2019) preprint.

%\bibitem{kukukwnapoposi} Kuznetsov, N., Kulczycki, T., Kwa\'{s}nicki, M., Nazarov, A., Poborchi, S., Polterovich, I., Siudeja, B., The legacy of Vladimir Andreevich Steklov, Notices Amer. Math. Soc. 61 (2014), 9-22.

%\bibitem{la2014} Lamberti, P.D., Steklov-type eigenvalues associated with best Sobolev trace constants: domain perturbation and overdetermined system, Complex Var. Elliptic Equ. 59 (2014), 309-323.



%\bibitem{lampro1} Lamberti, P.D., Provenzano, L., Viewing the Steklov eigenvalues of the Laplace operator as critical Neumann eigenvalues, Current Trends in Analysis and its Applications, 171-178, Trends Math., Birkh\"{a}user/Springer, Cham, Switzerland, 2015.

%\bibitem{lampro2} Lamberti, P.D., Provenzano, L., Neumann to Steklov eigenvalues: asymptotic and monotonicity results, Proc. Roy. Soc. Edinburgh Sect. A 147 (2017), 429-447.

%\bibitem{lala2002} Lamberti, P.D., Lanza de Cristoforis, M. An analyticity result for the dependence of multiple eigenvalues and eigenspaces of the Laplace operator upon perturbation of the domain. Glasg. Math. J. 44 (2002), no. 1, 29-43. 

%\bibitem{lala}  Lamberti, P.D., Lanza de Cristoforis, M., A real analyticity result for symmetric functions of the eigenvalues of a domain dependent Dirichlet problem for the Laplace operator. J. Nonlinear Convex Anal. 5 (2004) 19-42.



%\bibitem{lanza} Lanza de Cristoforis, M., Higher order differentiability properties of the composition and of the inversion operator. Indag. Math. (N.S.) 5 (1994), no. 4, 457-482. 


%\bibitem{li} Li, J., A literature survey of mathematical study of metamaterials, Int. J. Numer. Anal. Model. 13 (2016), 230-243.

%\bibitem {liuzou} Liu, H., Zou, J., {\em Zeros of the Bessel and spherical Bessel functions and their applications for uniqueness in inverse acoustic obstacle scattering}, IMA Journal of Applied Mathematics (2007) - Vol. 72, no. 6, 817-831.

%\bibitem{mccon} McConnell, A. J., {\em Application of tensor analysis}, Dover Publications, Inc., New York, 1957.


%\bibitem{pro} Provenzano, L., Stubbe, J., Weyl-type bounds for Steklov eigenvalues, J. Spectr. Theory 9 (2019), 349-377.

%\bibitem{sihvola} Sihvola, A., Metamaterials in electromagnetics,  Metamaterials 1 (2007), 2-11.


%\bibitem{terouh} ter Elst, A. F. M., Ouhabaz, E. M., Convergence of the Dirichlet-to-Neumann operator on varying domains, Operator Semigroups Meet Complex Analysis, Harmonic Analysis and Mathematical Physics, 147-154, Oper. Theory Adv. Appl., 250, Birkh\"{a}user/Springer, Cham, Switzerland, 2015.


%\bibitem{web} Weber, C., A local compactness theorem for Maxwell's equations, Math. Methods Appl. Sci. 2 (1980), 12-25.



\end{thebibliography}
\end{document}